\documentclass[11pt]{article}
\usepackage{graphicx}
\usepackage{indentfirst, latexsym, bm,amssymb}
\usepackage{bbding}

\usepackage{amsfonts}
\usepackage{amsmath}
\usepackage{amsthm}
\usepackage{amscd}
\usepackage{amsbsy}
\usepackage[colorlinks=true]{hyperref}
\advance\textwidth by +1.0in \advance\textheight by +1.0in
\advance\oddsidemargin by -0.5in \advance\evensidemargin by -1.0in
\advance\topmargin by -0.5in
\newtheorem {theorem} {Theorem}
\newtheorem {proposition} [theorem]{Proposition}
\newtheorem {corollary} [theorem]{Corollary}
\newtheorem {lemma}  [theorem]{Lemma}

\newtheorem {remark} {Remark}
\newtheorem*{TheoremA}{Theorem A}
\newtheorem*{TheoremB}{Theorem B}
\newtheorem*{TheoremC}{Theorem C}
\newtheorem*{TheoremD}{Theorem D}

\newcommand{\X}{{\mathcal{\bf} X}}

\newcommand{\R}{{\mathbb R}}

\usepackage{CJK}
\newtheorem{Defi}[theorem]{Definition}

\title{\large\bf On the Patterson-Sullivan measure for geodesic flows on rank $1$ manifolds without focal points}

\author{Fei Liu \thanks{College of Mathematics and System Science, Shandong University of Science and Technology, Qingdao, 266590, P.R. China.
e-mail:  liufei@math.pku.edu.cn.}
\and Fang Wang
\thanks{School of Mathematical Sciences, Capital Normal University, Beijing, 100048, China; and Beijing Center for Mathematics and Information
Interdisciplinary Sciences (BCMIIS), Beijing 100048, P.R. China. e-mail: fangwang@cnu.edu.cn.}
\and Weisheng Wu
\thanks{Department of Applied Mathematics, College of Science, China Agricultural University, Beijing, 100083, P.R. China. e-mail: wuweisheng@cau.edu.cn}}
\date{\today}

\begin{document}
\maketitle

\begin{abstract}
In this article, we consider the geodesic flow on a compact rank $1$ Riemannian manifold $M$ without focal points, whose universal cover is denoted by $X$. On the ideal boundary $X(\infty)$ of $X$, we show the existence and uniqueness of the Busemann density, which is realized via the Patterson-Sullivan measure. Based on the the Patterson-Sullivan measure, we show that the geodesic flow on $M$ has a unique invariant measure of maximal entropy. We also obtain the asymptotic growth rate of the volume of geodesic spheres in $X$ and the growth rate of the number of closed geodesics on $M$. These results generalize the work of Margulis and Knieper in the case of negative and nonpositive curvature respectively.
\\


\end{abstract}


\section{\bf Introduction and main results}\label{intro}

\setcounter{section}{1}
\setcounter{equation}{0}\setcounter{theorem}{0}

In this article, we study the geodesic flow on a connected closed (compact and having no boundary) rank $1$ manifold without focal points. We  consider the invariant measure of maximal entropy of the geodesic flow. The uniqueness of such a measure for the geodesic flow on compact manifolds with nonpositive curvature is an important topic in the theory of geodesic flows, which was previously studied by R.~Bowen (cf.~\cite{Bo2}) and G.~Margulis (cf.~\cite{Ma0}) in the case of negative curvature, and then proved by G.~Knieper for the geodesic flow on a rank $1$ manifold with nonpositive curvature (cf.~\cite{Kn1}). By following the ideas in \cite{Kn1}, we extend Knieper's result to rank $1$ manifolds without focal points and prove the existence and uniqueness of the invariant measure of maximal entropy. We will also discuss some related topics, including the distribution of periodic geodesics on a rank $1$ manifold without focal points and the asymptotic growth of the volume of geodesic spheres in the universal cover of the manifold.

Suppose that $(M,g)$ is a $C^{\infty}$ connected compact $n$-dimensional Riemannian manifold,
where $g$ is a Riemannian metric. For each $p \in M$ and $v \in T_{p}M$,
let $\gamma_{v}$ be the unique geodesic satisfying the initial conditions $\gamma_v(0)=p$ and $\gamma'_v(0)=v$. The geodesic flow $\phi=(\phi^{t})_{t\in\mathbb{R}}$ (generated by the Riemannian metric $g$) on the unit tangent bundle $SM$ is defined as:
\[
\phi^{t}: SM \rightarrow SM, \qquad (p,v) \mapsto
(\gamma_{v}(t),\gamma'_{v}(t)),\ \ \ \ \forall\ t\in \R .
\]
Without special indication, all geodesics we are considering in this paper are the geodesics with unit speed.

A $\phi$-invariant probability measure $\mu$ is called the measure of maximal entropy if $h_{\mu}(\phi)\geq h_{\nu}(\phi)$ for any $\phi$-invariant probability measure $\nu$. By the variational principle, $h_{\mu}(\phi)=h_{\text{top}}(g)$, where $h_{\text{top}}(g)$ denotes the topological entropy of the geodesic flow on $SM$. In 1970, Margulis constructed a measure of maximal entropy for the geodesic flow on a compact Riemannian manifold of variable negative curvature, in his dissertation at Moscow State University. His results were first published in \cite{Ma0} as a short announcement. Later in \cite{Ma1} he gave more details, and the whole proofs were published eventually in \cite{Ma2}. Using different methods, Bowen also constructed a maximal entropy measure in \cite{Bo1} in 1972, for hyperbolic flows, of which the geodesic flow on a compact manifold of negative curvature is a primary example. Later in \cite{Bo2} he proved that in this case (negative curvature), the maximal entropy measure of the geodesic flow is unique. Therefore, the maximal entropy measures for the geodesic flows constructed by Bowen and Margulis are eventually the same one. Thus we call this measure the \emph{Bowen-Margulis measure}.

In 1984, A.~Katok  conjectured that the geodesic flow on a compact rank $1$ manifold of nonpositive curvature admits a unique invariant maximal entropy measure (cf. \cite{BuKa}). Here, a rank $1$ geodesic is a geodesic which does not have a parallel perpendicular Jacobi field, and a rank $1$ manifold is a Riemannian manifold which admits a rank $1$ geodesic. More precisely, we present the definition of rank in the following:

\begin{Defi}
For each $v \in SM$, we define \emph{\text{rank}($v$)} to be the dimension of the vector space of parallel Jacobi fields along the geodesic $\gamma_{v}$, and \emph{\text{rank}($M$):=$\min\{$\text{rank}$(v) \mid v \in SM\}$}. For a geodesic $\gamma$ we define \emph{\text{rank}($\gamma$)}=\emph{\text{rank}($\gamma'(t)$)}, $\forall\ t\in \mathbb{R}$.
\end{Defi}

If $M$ is a rank $1$ manifold, the unit tangent bundle $SM$ splits into two invariant subsets: the regular set $\text{reg}:= \{v\in SM \mid \text{rank}(v)=1\}$, and its complement $\text{sing}:= SM \setminus \text{reg}$, which is called the singular set.
It is still not known if the singular set has zero Liouville volume. This is a wide open problem since 1980's, and the positive answer will imply the ergodicity of the geodesic flow on rank $1$ manifolds of nonpositive curvature or without focal points w.r.t. Liouville measure (cf. \cite{Wu1,WLW}).
We also remark that if $M$ has rank greater than $1$ and of nonpositive curvature, the celebrated higher rank rigidity theorem, which was established independently by W. Ballmann and K. Burns-R. Spatzier in \cite{Ba} and \cite{BS2} based on the work in W. Ballmann-M. Brin-P. Eberlein (cf.~\cite{BBE}) and W. Ballmann-M. Brin-R. Spatzier (cf.~\cite{BBS}), asserts that the universal cover $X$ of $M$ is a flat Euclidean space, a symmetric space of noncompact type, a space of rank 1, or a product of the above types. The higher rank rigidity theorem is also extended to the manifolds without focal points (cf. \cite{Wat}) and a class of Finsler manifolds (cf. \cite{Wu2}).

Katok's conjecture was eventually proved by Knieper in \cite{Kn1}. In his proof, instead of considering the measures supported on closed geodesics as in \cite{Bo1}, Knieper studied the so-called Busemann density defined by using the Poincar\'e series,
and used it to construct the measure of maximal entropy (we call it \emph{Knieper measure}).
Knieper's innovative work points out a new way to study the measure of maximal entropy for geodesic flows.
An immediate question is that does this result hold for more general situations?
The manifolds without focal points / conjugate points are usually considered to be the natural extension of
the conception of manifolds with nonpositive curvature.

\begin{Defi}
Let $\gamma$ be a geodesic on $(M,g)$. A pair of distinct points $p=\gamma(t_{1})$ and $q=\gamma(t_{2})$ are called \emph{focal} if there is a Jacobi field $J$ along $\gamma$ such that $J(t_{1})=0$, $J'(t_{1})\neq 0$ and $\frac{d}{dt}\| J(t)\|^{2}\mid_{t=t_{2}}=0$; $p=\gamma(t_{1})$ and $q=\gamma(t_{2})$ are called \emph{conjugate} if there is a non-identically-zero Jacobi field $J$ along $\gamma$ such that $J(t_{1})=0=J(t_{2})$.

A compact Riemannian manifold $(M,g)$ is called a manifold \emph{without focal points / without conjugate points} if there is no focal points / conjugate points on any geodesic in $(M,g)$.
\end{Defi}

From the definitions, it is easy to see that if a manifold has no focal points then it has no conjugate points. It is well known that all manifolds with nonpositive curvature have no focal points. In this sense, manifolds without focal / conjugate points can be regarded as a generalization of the manifolds with nonpositive curvature. This generalization is non-trivial,
since it is easy to construct a compact manifold without focal points whose curvature is not everywhere nonpositive (cf.~\cite{Gu}).

In this paper we generalize Knieper's work and prove the existence and uniqueness of the measure of maximal entropy
on a compact rank $1$ manifolds without focal points, by employing Knieper's idea to the more general situation.
This is the following theorem:

\begin{TheoremA}\label{ThmA}
Suppose $(M,g)$ is a smooth compact rank $1$ Riemannian manifold without focal points, then
\begin{enumerate}
\item the geodesic flow on $(M,g)$ has a unique measure of maximal entropy;
\item $h_{\text{top}}(\phi|_{\text{sing}})< h_{\text{top}}(g)$, where $h_{\text{top}}(g)$ is the topological entropy of the geodesic flow on $SM$,
and $h_{\text{top}}(\phi|_{\text{sing}})$ is the topological entropy of the geodesic flow restricted on $\text{sing}$.
\end{enumerate}
\end{TheoremA}
\begin{remark}
In a recent preprint \cite{GR}, K.~Gelfert and R.~Ruggiero proved that the geodesic flow on a compact surface without focal points has a unique measure of maximal entropy. Their idea is to consider the time-preserving semi-conjugacy from the geodesic flow to a continuous expansive flow with a local product structure, which has a unique measure of maximal entropy.

The property $h_{\text{top}}(\phi|_{\text{sing}})< h_{\text{top}}(g)$ is usually called the ``entropy gap" for the geodesic flows.  It is obtained for the geodesic flows on compact rank $1$ manifolds of nonpositve curvature by Knieper in \cite{Kn1}. Recently, this is generalized to a ``pressure gap"
for the geodesic flows on compact rank $1$ manifolds of nonpositve curvature, by using a different method by K.~Burns-V.~Climenhaga-T.~Fisher-D.~J.~Thompson in \cite{BCTT}. Our result shows the entropy gap in the situation of no focal points.
\end{remark}

The key ingredient in the proof is the Patterson-Sullivan measure constructed from the Poincar\'{e} series.
Let $M=X/\Gamma$ be a compact rank $1$ manifold without focal points, where $X$ is the universal cover of $M$ and $\Gamma$ is a discrete subgroup of the isometry group $\text{Iso}(X)$. We will extend the Patterson-Sullivan construction to the rank $1$ manifolds without focal points, and show that the Patterson-Sullivan measure is essentially the unique Busemann density. We will discuss the notion of Busemann density and the Patterson-Sullivan construction in Subsection \ref{Busemann1}.

We note that in nonpositive curvature case, Knieper's method relies on the convexity properties of various distance functions. In no focal points case, convexity property are replaced by sort of the no maxima property. Nevertheless, some results such as flat strip lemma remain true.
In the course of the proof of Theorem A, we establish all the geometric properties needed in Knieper's method,
more precisely Propositions \ref{pro1}, \ref{pro2}, \ref{pro3}-\ref{pro6} and \ref{pro8}.
In Propositions \ref{pro1} and \ref{pro2},
we prove the continuity property of geodesics at infinity and describe the property of cones of simply connected
manifolds without focal points;
in Proposition \ref{pro3} we show that compact manifolds without focal points satisfy the duality condition,
a property that has many important applications, including the rank rigidity theory in nonpositive curvature(cf.~\cite{Ba1, Eb1});
in Propositions \ref{pro4}-\ref{pro6}, we establish the existence of the connecting geodesics between the neighborhoods of the endpoints of a rank $1$ geodesic,
which is crucial for the semi-local properties in the no focal case appearing in Subsection \ref{proj};
in Proposition \ref{pro8}, we prove that the fundamental group acts minimally on the ideal boundary,
which is extraordinarily important in our discussion on the Patterson-Sullivan measure (see Proposition \ref{pro12}).
These properties are important for the study of the dynamics of the geodesic flows on rank $1$ manifolds with no focal points, and hence of independent interests.
Recently in \cite{LZ} we used these properties to prove that the geodesic flows on compact rank $1$ manifolds without focal points
are topologically transitive.

\begin{TheoremB}\label{ThmB}
Let $M=X/\Gamma$ be a compact rank $1$ manifold without focal points, then up to a multiplicative constant, the Busemann density is unique, i.e., the Patterson-Sullivan measure is the unique Busemann density.
\end{TheoremB}

The Busemann density was first constructed by Patterson in \cite{Pat}. Uniqueness of the Busemann density was established by Sullivan (cf.~\cite{Su}) for compact manifolds of negative curvature and by Knieper (cf.~\cite{Kn0})
for compact rank $1$ manifolds of nonpositive curvature. To our knowledge, it is the first time the Busemann density is considered for compact rank $1$ manifold without focal points, in Theorem B. We also show that the critical exponent of the Poincar\'{e} series of the co-compact group $\Gamma\subset \text{Iso}(X)$ with $M=X\setminus \Gamma$ coincide with the topological entropy of the geodesic flow on $SM$.

Since the Poincar\'{e} series is of divergent type, the Patterson-Sullivan measure can be used to construct a finite measure on $SX$ which is invariant under the geodesic flow and $\Gamma$-action. The main work in this article is to show the invariant measure under the geodesic flow on $SM$ constructed in this way is the unique measure of maximal entropy, proving Theorem A.

Furthermore, the power of the Patterson-Sullivan measure is not only limited to this problem. It can also be used to investigate the asymptotic geometry of the rank 1 compact manifolds without focal points. Let $B(x,r)$ be a geodesic ball in $X$ about $x$ with radius $r>0$, and $S(x,r)=\partial B(x,r)$ be a geodesic sphere in $X$. Let $h=h_{\text{top}}(g)$ be the topological entropy of the geodesic flow, which will be shown to coincide with the critical exponent of the Poincar\'{e} series of the co-compact group $\Gamma\subset \text{Iso}(X)$ with $M=X\setminus \Gamma$, as mentioned above. The classical results of Manning (cf.~\cite{Man}) and Freire-Ma\~n\'e (cf.~\cite{FrMa}) show
$$h=\lim_{r\to \infty}\frac{\log \text{Vol}(B(x,r))}{r}.$$
In \cite{Ma0}, Margulis obtained a finer estimate in the case of negative curvature: there are constants $a>0$ and $r_0>0$ such that
$$\frac{1}{a} \leq \frac{\text{Vol}(S(x,r))}{e^{h \cdot r}} \leq a,~\forall~r>r_0.$$
It is Knieper who applied the Patterson-Sullivan measure to obtain a same result on the asymptotic growth of the volume of the geodesic spheres
for compact rank $1$ manifolds of nonpositive curvature (cf. \cite{Kn0}). His method can be extended to compact rank $1$ manifolds without focal points:

\begin{TheoremC}\label{ThmC}
Suppose $(M,g)$ is a compact rank $1$ manifold without focal points and $X$ is its universal cover. Let $x\in X$ be an arbitrary point and $S(x,r)$ be the sphere centered at $x$ with  radius $r$. Then there are constants $a>0$ and $r_0>0$ such that $$\frac{1}{a} \leq \frac{\text{Vol}(S(x,r))}{e^{h \cdot r}} \leq a,~\forall~r>r_0,$$ where $h=h_{\text{top}}(g)$ is the critical exponent of the Poincar\'{e} series of the group $\Gamma\subset Iso(X)$ with $M=X\setminus \Gamma$.
\end{TheoremC}

A straightforward application of Theorem A is estimating the distribution of the regular and singular primitive closed geodesics with a given upper-bound of the periods on a compact rank $1$ manifold without focal points.  Here, we say a closed geodesic is primitive, if it is not an iterate of another closed geodesic. If a closed geodesic staying in the regular set, we call it a regular closed geodesic, otherwise we call it a singular closed geodesics. Let $\mathcal{P}(M)$ be a maximal set of geometrically distinct primitive closed geodesics which represent different free homotopy classes. For each constant $t>0$, we use $\mathcal{P}(t)$ to denote the set of geodesics in $\mathcal{P}(M)$ with least periods less than or equal to $t$. Note that since the flat strip theorem holds for manifolds without focal points, the least periods of all closed geodesics are equal in a given free homotopy class. Let $\mathcal{P}_{\text{reg}}(t)\subset \mathcal{P}(t)$ be the subset of regular closed geodesics in $\mathcal{P}(t)$, and $\mathcal{P}_{\text{sing}}(t)= \mathcal{P}(t) \backslash \mathcal{P}_{\text{reg}}(t)$ be the subset of singular closed geodesics in $\mathcal{P}(t)$. Denote by $P_{\text{reg}}(t):= \sharp \mathcal{P}_{\text{reg}}(t)$ and $P_{\text{sing}}(t):= \sharp \mathcal{P}_{\text{sing}}(t)$. Then Theorem A (2) and Theorem C imply the following result:

\begin{TheoremD}\label{ThmD}
Let $(M,g)$ be a compact rank $1$ Riemannian manifold without focal points. Then there exist $a>0$ and $t_1>0$ such that for all $t>t_1$,
$$\frac{e^{ht}}{at}\leq P_{\text{reg}}(t)\leq a e^{ht}.$$
Moreover, there exist positive constants $\epsilon$ and $t_{2}$ such that
$$\frac{P_{\text{sing}}(t)}{P_{\text{reg}}(t)}\leq e^{-\epsilon t}, ~~t> t_{2}.$$
\end{TheoremD}

This paper is organized in the following way: In Section \ref{geometric}, we study the geometric properties of rank $1$ manifolds without focal points, which will be frequently used in our subsequent discussion.
We present the Patterson-Sullivan construction for rank $1$ manifolds without focal points and prove Theorem B in Section \ref{Busemann}. A key technical result about the Patterson-Sullivan measure is prepared in Subsection \ref{proj}. Then Theorem C is proved in Sections \ref{volume}. In the last two Sections \ref{maximal} and \ref{closedgeodesic}, we prove Theorems A and D. We should point out that our main arguments follow the ideas of Knieper's work in \cite{Kn0} and \cite{Kn1}, which are showed to work for geodesic flows on rank $1$ manifolds without focal points, based on the geometric properties we establish in this paper.

\section{\bf Geometric properties of rank $1$ manifolds without focal points}\label{geometric}
\setcounter{section}{2}
\setcounter{equation}{0}\setcounter{theorem}{0}

In this section, we present some geometric results on rank $1$ manifolds without focal points, which will be used in the following discussions. We remark that, although in this paper we only discuss the rank $1$ manifolds without focal points, some of the results are also valid in more general situations.

Let $X$ be the universal covering manifold of $M$ and $d$ is the distance function on $X$ induced by the lifted Riemannian metric $\tilde{g}$ on $X$. Suppose $h_{1}$ and $h_{2}$ are both geodesics in $X$. We call $h_{1}$ and $h_{2}$ are \emph{positively asymptotic} if there is a positive number $C > 0$ such that
\begin{equation}\label{e1}
d(h_{1}(t),h_{2}(t)) \leq C, ~~\forall~ t \geq 0.
\end{equation}
We say $h_{1}$ and $h_{2}$ are \emph{negatively asymptotic} if \eqref{e1} holds for all $t \leq 0$. $h_{1}$ and $h_{2}$ are said to be \emph{biasymptotic} if they are both positively asymptotic and negatively asymptotic.
The relation of (positive / negative) asymptoticity is an equivalence relation between geodesics on $X$. The class of geodesics that are positively / negatively asymptotic to a given geodesic $\gamma$ is denoted by $\gamma(+\infty)$ / $\gamma(-\infty)$ respectively. We call them \emph{points at infinity}. Obviously, $\gamma_{v}(-\infty)=\gamma_{-v}(+\infty)$. We use $X(\infty)$ to denote the set of all points at infinity,
and call it the \emph{boundary at infinity}, or the \emph{ideal boundary}.

Let $\overline{X}=X \cup X(\infty)$.
For each point $p \in X$ and $v\in S_{p}X$, each points $x, y \in \overline{X}-\{p\}$, positive numbers $\epsilon$ and $r$,
and subset $A\subset\overline{X}-\{p\}$,
we define the following notations:
\begin{itemize}
\item{} ~~$\gamma_{p,x}$ is the geodesic from $p$ to $x$ and satisfies $\gamma_{p,x}(0)=p$.
\item{} ~~$\measuredangle_{p}(x,y)=\measuredangle(\gamma'_{p,x}(0),\gamma'_{p,y}(0))$.
\item{} ~~$\measuredangle_{p}(A)=\sup \{\measuredangle_{p}(a,b)\mid a,\ b \in A \}$.
\item{} ~~$\measuredangle(v,x)=\measuredangle(v,\gamma'_{p,x}(0))$.
\item{} ~~$C(v,\epsilon)=\{a \in \overline{X}-\{p\}\mid \measuredangle(v,a)< \epsilon\}$.
\item{} ~~$C_{\epsilon}(v) =C(v,\epsilon)\cap X(\infty)= \{\gamma_{w}(+\infty)\mid w \in  S_{x}X, \angle(v,w)<\epsilon\}$.
\item{} ~~$TC(v,\epsilon,r)= \{q \in \overline{X}\mid \measuredangle_{p}(\gamma_{v}(+\infty),q)< \epsilon\} - \{q \in \overline{X}\mid d(p,q)\leq r\}$.
\end{itemize}

$TC(v,\epsilon,r)$ is called the \emph{truncated cone} with axis $v$ and angle $\epsilon$. Obviously $\gamma_{v}(+\infty) \in TC(v,\epsilon,r)$. There is a unique topology $\tau$ on $\overline{X}$ such that for each $\xi \in X(\infty)$
the set of truncated cones containing $\xi$ forms a local basis for $\tau$ at $\xi$. This topology is usually called the \emph{cone topology}. Under this topology, $\overline{X}$ is homeomorphic to the closed unit ball in $\mathbb{R}^{\text{dim}(X)}$, and the ideal boundary $X(\infty)$ is homeomorphic to the unit sphere $\mathbb{S}^{\text{dim}(X)-1}$.
For more details about the cone topology, see \cite{EbON} and \cite{Eb1}.

The following results of manifolds without focal points are well known. We state them as two lemmas here, for these results will be used very often in this paper.

\begin{lemma}[cf.~O'Sullivan \cite{Os}]\label{OS}
Let $X$ be a simply connected Riemannian manifold without focal points.
\begin{enumerate}
\item Let $h_{1}$ and $h_{2}$ be a pair of distinct geodesic rays starting from a same point in $X$. Then for $t>0$, the function $d_{h_1,h_2}(t):=d(h_{1}(t),h_{2}(t))$ is strictly increasing, and $\lim_{t\rightarrow+\infty}d_{h_1,h_2}(t)=+\infty$.
\item Let $h_{1}$ and $h_{2}$ be a pair of positively asymptotic geodesics. Then the distance function $d_{h_1,h_2}(t)=d(h_{1}(t),h_{2}(t))$ is non-increasing for all $t \in \mathbb{R}$.
\end{enumerate}
\end{lemma}

\begin{remark}
Lemma \ref{OS} implies that on a simply connected Riemannian manifold without focal points, two distinct geodesics can not be positively or negatively asymptotic if they cross at some point on $X$.
\end{remark}

\begin{lemma}[cf. for example Proposition 2.8 in Katok \cite{Ka}]\label{Ka}
Let $h_{1}$ and $h_{2}$ be two distinct geodesics on a simply connected Riemannian manifold without focal points,
then for any $a \in \mathbb{R}$ and $t\in[0,a]$, the following inequality holds:
$$d(h_{1}(t),h_{2}(t)) \leq d(h_{1}(0),h_{2}(0))+d(h_{1}(a),h_{2}(a)).$$
\end{lemma}

Our first result in this section is the property of \emph{the continuity to infinity}. We define a metric $\delta$ on $SX$ by $$\delta(v,w):=d(\pi v, \pi w)+d(\exp v , \exp w).$$ It can be shown that this distance is equivalent to the distance induced by the standard Sasaki metric. The following proposition shows the continuity of the map $(v,t) \mapsto \gamma_{v}(t)$ from $SX \times [-\infty,+\infty]$ to $\overline{X}$ with respect to the metric $\delta$,
on a simply connected manifold without focal points. It is the no focal points version of Proposition 2.13 of \cite{EbON}.

\begin{proposition}[Continuity to infinity]\label{pro1}
 Let $X$ be a simply connected manifold without focal points, the map
 $$\Psi: SX \times [-\infty,+\infty] \rightarrow \overline{X}=X\cup X(\infty),$$
 $$(v,t) \mapsto \gamma_{v}(t).$$
 is continuous.
\end{proposition}
\begin{proof}
The continuity is obvious when  $t\in (-\infty,+\infty)$, since the geodesic flow is continuous. First we show that if $v_{n}\rightarrow v \in SX$ and $t_{n}\rightarrow +\infty$, then $\gamma_{v_{n}}(t_{n})\rightarrow \gamma_{v}(+\infty)$.

Let $p=\pi(v),\ p_{1}=\pi(v_{1}),\ \cdots,\ p_{n}=\pi(v_{n}),\ \cdots\in X$. For each $n\in\mathbb{Z}^+$, let $w_{n}\in S_pX$ be the unique initial vector at $p$ such that $\gamma_{w_{n}}(s_{n}) = \gamma_{v_{n}}(t_{n})$ for some positive number $s_{n}>0$. By the triangle inequality, we have
\begin{equation}\label{eq1}
|s_{n}-t_{n}|\leq d(p,p_{n}).
\end{equation}
Thus, $s_{n}\rightarrow +\infty$ as $t_{n}\rightarrow +\infty$.

Let $q_{n}=\gamma_{v_{n}}(t_{n})=\gamma_{w_{n}}(s_{n})$. Now we use $\gamma_{v^{-1}_{n}}$ to denote the geodesic from $q_{n}$ to $p$ with  $\gamma_{v^{-1}_{n}}(0)=q_n$,
and $\gamma_{q_{n}}$ to denote the geodesic from $q_{n}$ to $p_{n}$ with  $\gamma_{q_{n}}(0)=q_n$, respectively.
By Lemma \ref{OS} and \eqref{eq1}, we have
$$d(\gamma_{q_{n}}(t),\gamma_{v^{-1}_{n}}(t)) \leq d(p,p_{n})+|s_{n}-t_{n}|\leq 2d(p,p_{n}),  ~~t\in [0,s_{n}].$$
Thus
\begin{eqnarray*}
\delta(v_{n},w_{n})
&=& d(\pi v_{n},\pi w_{n})+d(\exp v_{n},\exp w_{n})\\
&=& d(p_{n},p)+ d(\gamma_{q_{n}}(s_{n}-1),\gamma_{v^{-1}_{n}}(t_{n}-1))\\
&\leq& 4 d(p_{n},p).
\end{eqnarray*}
By $v_{n}\rightarrow v$ and the above inequality, we get $w_{n}\rightarrow v$. This implies that for any $\epsilon>0$ there is an $N=N(\epsilon)>0$ such that if $n>N$, then $\gamma_{v_{n}}(t_{n})\in TC(v,\epsilon, \frac{s_{n}}{2})$. So it follows that, under the cone topology,
$\gamma_{v_{n}}(t_{n})=\gamma_{w_{n}}(s_{n})\rightarrow \gamma_{v}(+\infty)$.

Now we consider the continuity of $\Psi(\cdot,+\infty):SX\to X(\infty)$. Suppose $\{v_n\}\subset SX$ is a convergent sequence with $\lim_{n\to \infty}v_n=v\in S_pX$ for some $p\in X$. For each $\gamma_{v_n}(+\infty)\in X(\infty)$, there is a unique vector $w_n\in S_p X$ such that $\gamma_{w_n}(+\infty)=\gamma_{v_n}(+\infty)$. Similar as the argument in the above we can show that $w_n\to v$.
We know on the compact space $S_{p}X$, all the metrics are equivalent, thus the metric $\delta$ defined above
is equivalent to the usual angle metric. Therefore $\lim_{n\rightarrow +\infty}\theta_{n}=0$, where $\theta_{n}$ is the angle between $v$ and $w_{n}$.
Thus, passing to a subsequence if necessary, we can assume that $\theta_{n}< \frac{1}{n}$.
This means that $\gamma_{v_{n}}(+\infty)=\gamma_{w_{n}}(+\infty) \in TC(\gamma_{v}(+\infty),\frac{1}{n},n)$.
Since $$\bigcap_{n =1}^{\infty}TC(\gamma_{v}(+\infty),\frac{1}{n},n) = \gamma_{v}(+\infty),$$ we get that $\gamma_{v_{n}}(+\infty) \rightarrow \gamma_{v}(+\infty)$.
\end{proof}

The following proposition is a corollary of Proposition \ref{pro1}. It will be used in our discussion in the next section.

\begin{proposition}\label{pro2.4}
Let $X$ be a simply connected manifold without focal points. Then for each point $p \in X$, the map
$$\gamma_{+\infty}: S_{p}X\rightarrow X(\infty), ~~v\mapsto \gamma_{+\infty}(v):=\gamma_{v}(+\infty)$$
is a homeomorphism.
\end{proposition}
\begin{proof}
By Lemma \ref{OS}(1) we know that the map $\gamma_{+\infty}$ is bijective. Proposition \ref{pro1} implies that $\gamma_{+\infty}$ is a continuous map from $S_{p}X$ to $X(\infty)$.
Thus it suffices to show that the inverse map $\gamma^{-1}_{+\infty}:X(\infty)\rightarrow S_{p}X$ is continuous.

For each $\xi \in X(\infty)$ and an arbitrary sequence $\{\xi_{n}\}\subset X(\infty)$ with $\lim_{n\rightarrow +\infty}\xi_{n}=\xi$,
we want to show that $\lim_{n\rightarrow +\infty}\gamma^{-1}_{+\infty}(\xi_{n})=\gamma^{-1}_{+\infty}(\xi)$.
Let $v_{n}=\gamma^{-1}_{+\infty}(\xi_{n}) \in S_{p}X$, $n=1,2,\cdots$, and $v=\gamma^{-1}_{+\infty}(\xi) \in S_{p}X$.
Passing to a subsequence if needed, we can assume that $\lim_{n\rightarrow +\infty}v_{n}=w$ for some $w\in S_{p}X$. By Proposition \ref{pro1},
we know that $$\gamma_{w}(+\infty)=\lim_{n\rightarrow +\infty}\gamma_{v_{n}}(+\infty)=\lim_{n\rightarrow +\infty}\xi_{n}=\xi.$$
Then Lemma \ref{OS}(1) implies that $w=v$, which means that the inverse map $\gamma^{-1}_{+\infty}$ is continuous.
We are done with the proof.
\end{proof}

For each $p \in X$, let $B_{p}=\{v \in T_{p}X \mid\ \| v\| \leq 1\} \subset T_{p}X$, which is the unit ball in $T_pX$.
Using the same method in the proof of Proposition \ref{pro2.4}, we can get the following result.
\begin{corollary}\label{cor2.5}
Let $X$ be a simply connected manifold without focal points. Then for each point $p \in X$, the map
\[
F: B_{p} \rightarrow \overline{X}, ~~v\mapsto F(v):=\left\{\begin{array}{ll}  \exp_{p}(\frac{v}{1-\parallel v\parallel})\qquad&\mbox{ if } \quad
\parallel v\parallel < 1,\\
\gamma_{v}(+\infty)\qquad  &\mbox{ if }\quad v \in S_{p}X,
\end{array}
\right.
\]
is a homeomorphism.
\end{corollary}

\begin{proposition}\label{pro2}
Let $X$ be a simply connected manifold without focal points. Then for any $v \in SX$, $R > 0$, and $\epsilon > 0$, there is a constant $L=L(v,\epsilon,R)$ such that for all $t > L$, $$B(\gamma_{v}(t),R)\subset C(v,\epsilon).$$
\end{proposition}
\begin{proof}
We prove this proposition by contradiction. Suppose that there is a vector $v \in SX$, and constants $\epsilon >0$ and $R > 0$ such that $\forall L > 0$, we can always find $t > L$ satisfying $$B(\gamma_{v}(t),R) \setminus C(v,\epsilon)\neq \emptyset.$$
Let $p$ be the footpoint of $v$, i.e. $v\in S_pX$. Take $L_{n}=n$, then $\exists\ t_{n}> L_{n}=n$, and $p_{n}\in B(\gamma_{v}(t_{n}),R)$ such that $$\measuredangle(v,p_{n})\geq \epsilon_{0},~~|d(p,p_{n})-t_{n}| < R.$$ i.e. we have
$$t_{n}-R<d(p,p_{n})<t_{n}+R.$$
Let $\gamma_{n}$ be the geodesic ray from $p$ to $p_{n}$ and $v_{n}=\gamma'_{n}(0)$. Then by Lemma \ref{OS}, it's easy to see that
\begin{equation}\label{eq2}
d(\gamma_{v}(t),\gamma_{n}(t))\leq 2R,~~t\in[0,t_{n}].
\end{equation}

Passing a subsequence of $\{v_{n}\}^{\infty}_{n=1}$ if necessary, we can assume $v_{\infty}=\lim_{n\rightarrow \infty}v_{n}$. Obviously $v_{\infty}\in S_pX$.
Since $\measuredangle(v,v_{\infty})\geq \epsilon_{0}$, we know $v\neq v_{\infty}$. However, by the inequality \eqref{eq2} and Proposition \ref{pro1}, we have $\gamma_{v_{\infty}}(+\infty)=\gamma_{v}(+\infty)$. This contradicts to Lemma \ref{OS} since both $\gamma_{v_{\infty}}$ and $\gamma_{v}$ start from $p$.

From the argument in the above, we know that for any $v \in SX$, $R > 0$, and $\epsilon > 0$, there is a constant $L$ such that for all $t > L$, $B(\gamma_{v}(t),R)\subset C(v,\epsilon)$. We are done with the proof.
\end{proof}

In the following, we prove the duality property on Riemannian manifolds without focal points. Before stating the result, we should introduce the concepts of \emph{recurrent point} and  \emph{nonwandering point} with respect to a discrete subgroup $\Gamma$ of isometry group $Iso(X)$ of $X$,
where $X$ is a simply connected manifold.

\begin{Defi}
We say $v \in SX$ is a \emph{recurrent point} of the geodesic flow $\phi^{t}:SX\rightarrow SX$ w.r.t.~$\Gamma \subset Iso(X)$, if there exist a sequence $t_n\rightarrow +\infty$ and a sequence $\{\alpha_n\}_{n\in\mathbb{Z}^+} \in \Gamma$ such that
$$(d\alpha_{n}\circ\phi^{t_n})(v)\rightarrow v.$$
We say $v \in SX$ is a \emph{nonwandering point} of geodesic flow $\phi^{t}:SX\rightarrow SX$ w.r.t.~$\Gamma \subset Iso(X)$, if for any neighborhood $U$ of $v$ in $SX$, and any number $A>0$, there exists a number $t\geq A$ and $\alpha \in \Gamma$ such that
$$(d\alpha\circ\phi^{t})(U)\cap U \neq \emptyset.$$
We use $\Omega(\Gamma)$ to denote the set of all nonwandering points of geodesic flow $\phi^{t}:SX\rightarrow SX$.
\end{Defi}

This definition of $\Omega(\Gamma)$ was introduced by W. Ballmann in \cite{Ba2}.
Let $\mathfrak{C}:X \rightarrow M$ be the universal covering map. Then $d\mathfrak{C}(\Omega(\Gamma))$ is the nonwandering set
of the geodesic flow $\phi^{t}:SM \rightarrow SM$ in the common sense in dynamical systems (\cite{Eb1, KH, Pa}).
Moreover, it's easy to verify that for any $\alpha\in \Gamma$, the following diagrams
\begin{equation}\label{commutediagram}
\begin{array}[c]{ccc}
SX &\stackrel{\phi^{t}}{\longrightarrow}&SX\\
\downarrow\scriptstyle{d\alpha}&&\downarrow\scriptstyle{d\alpha}\\
SX &\stackrel{\phi^{t}}{\longrightarrow}&SX\\
\end{array},
\begin{array}[c]{ccc}
SX &\stackrel{\phi^{t}}{\longrightarrow}&SX\\
\downarrow\scriptstyle{d\mathfrak{C}}&&\downarrow\scriptstyle{d\mathfrak{C}}\\
SM &\stackrel{\phi^{t}}{\longrightarrow}&SM\\
\end{array}
\end{equation}
commute. Here we use $\phi^{t}$ to denote both the geodesic flow on $SX$ and the geodesic flow on $SM$.

\begin{proposition}[Duality property]\label{pro3}
Let $X$ be a simply connected manifold without focal points and $\Gamma$ be a discrete subgroup of the isometry group $Iso(X)$.
If $X/\Gamma = M$ is a compact manifold, then for all $v\in SX$, $\gamma_{v}(-\infty)$ and $\gamma_{v}(+\infty)$ are \emph{$\Gamma$-dual}, i.e., $\exists\ \{\alpha_{n}\}^{+\infty}_{n=1}\subset \Gamma$ such that for any $p\in X$
\begin{equation}\label{duality}
\alpha^{-1}_{n}(p)\rightarrow \gamma_{v}(+\infty),~\&~\alpha_{n}(p)\rightarrow \gamma_{v}(-\infty).
\end{equation}
\end{proposition}
\begin{proof}
First we show that $\Omega(\Gamma)=SX$. We prove this by contradiction. Assume that there exists a vector $v \in SX$ such that $v \notin \Omega(\Gamma)$.
Then we can find a neighborhood $U$ of $v$ in $SX$ with $\mbox{diag}(U) < \mbox{Inj}(M)$, where $\mbox{Inj}(M)$ is the injective radius of $M$,
and a constant $A > 0$ such that for any $t \geq A$ and $\alpha \in \Gamma$,
$(d\alpha\circ\phi^{t})(U)\cap U = \emptyset$. Thus for any $\alpha \in \Gamma$ and $t_{2}> t_{1}\geq A$ satisfying $t_{2} - t_{1}\geq A$,
we have $\phi^{t_{2}-t_{1}}(U)\cap d\alpha(U) = \emptyset$. Using the first commutative diagram in (\ref{commutediagram}), we get
\begin{equation}\label{eq2.6}
\phi^{t_{2}}(U)\cap (d\alpha \circ \phi^{t_{1}})(U) = \emptyset, ~~\forall \alpha \in \Gamma.
\end{equation}
Together with the second diagram in (\ref{commutediagram}), we know that
\begin{equation}\label{eq2.61}
\phi^{t_{2}}(d\mathfrak{C}(U)) \cap \phi^{t_{1}}(d\mathfrak{C}(U))= \emptyset,  ~~\forall t_{2}> t_{1}\geq A  ~~\mbox{with} ~~t_{2} - t_{1}\geq A.
\end{equation}
Let $V:=d\mathfrak{C}(U)\subset SM$.
Since $\mathfrak{C}$ is an $\mbox{Inj}(M)$-isometry (see for example \cite{Pa} remark 3.37) and the geodesic flow preserves the Liouville measure $\lambda$ on $SM$,
we know that for any $t\in \mathbb{R}$,
\begin{equation}\label{eq2.62}
\lambda(\phi^{t}(V))=\lambda(V)>0.
\end{equation}
(\ref{eq2.61}) and (\ref{eq2.62}) imply that
$\phi^{nA}(V) ~(n=1,2,3,...)$ are pairwisely disjoint sets with equal positive Liouville measure in $SM$.
While $SM$ is compact, $\lambda(SM)=1 < +\infty$, this is impossible. From the argument in the above, we can conclude that $\Omega(\Gamma)=SX$.

Now for all $v\in SX$, we can find sequences $\{\alpha_{n}\}^{+\infty}_{n=1}\subset \Gamma$, $v_{n}\rightarrow v$ and $t_{n}\rightarrow +\infty$ such that
$$d\alpha_{n}\circ \phi^{t_{n}}(v_{n})\rightarrow v.$$
This implies that $\alpha_{n}\gamma_{v_{n}}(t_{n})\rightarrow \gamma_{v}(0)$, and then
$$d(\alpha^{-1}_{n}\gamma_{v}(0),\gamma_{v_{n}}(t_{n}))\rightarrow 0.$$
Therefore by the property of continuity to infinity (Proposition \ref{pro1}), we have
$$\lim_{n\rightarrow\infty}\alpha^{-1}_{n}\gamma_{v}(0)= \lim_{n\rightarrow \infty}\gamma_{v_{n}}(t_{n})=\gamma_{v}(+\infty).$$
We have done the proof of the first limit.

For the second one, let $\gamma_{n}(t)= \alpha_{n}\gamma_{v_{n}}(t_{n}-t)$,
then
\begin{eqnarray*}
\gamma'_{n}(0)
&=& \frac{d}{dt}\mid_{t=0}\alpha_{n}\gamma_{v_{n}}(t_{n}-t)\\
&=& d\alpha_{n}(-\phi^{t_{n}}(v_{n}))\rightarrow -v.
\end{eqnarray*}
By Proposition \ref{pro1}, we have $\gamma_{n}(t_{n})\rightarrow \gamma_{-v}(+\infty)=\gamma_{v}(-\infty)$.
Therefore, $\gamma_{n}(t_{n})=\alpha_{n}\gamma_{v_{n}}(t_{n}-t_{n})=\alpha_{n}\gamma_{v_{n}}(0)$, which implies
$$\alpha_{n}\gamma_{v}(0)\rightarrow \gamma_{v}(-\infty).$$

Take $p=\gamma_v(0)$, then (\ref{duality}) is valid.
Proposition \ref{pro2} implies that if (\ref{duality}) holds for one point,
then it holds for all points in $X$. This complete the proof of Proposition \ref{pro3}.
\end{proof}

The following proposition was first discovered by J. Watkins in \cite{Wat}. Here we give a new proof by using a different method. This proposition can be used as a criterion to distinguish a vector of higher rank  on a simply connected manifold without focal points.

\begin{proposition}\label{pro4}
Suppose $X$ is a simply connected manifold without focal points and $v\in SX$. If there is a constant $c>0$ such that for all $k \in \mathbb{Z}^+$, we can always find a pair of points in the cones:
$$p_{k}\in C(-v,\frac{1}{k}), ~~q_{k}\in C(v,\frac{1}{k}),$$ with $d(\gamma_{v}(0),\gamma_{p_{k},q_{k}})\geq c$, then $\text{rank}(v)\geq 2$.
\end{proposition}
\begin{proof}
First of all, we will show that under the assumption of this proposition, both of the following equalities hold:
\begin{equation}\label{eq3}
\lim_{k\rightarrow +\infty}d(\gamma_{v}(0),p_{k})=+\infty,
\end{equation}
\begin{equation}\label{eq4}
\lim_{k\rightarrow +\infty}d(\gamma_{v}(0),q_{k})=+\infty.
\end{equation}

As a first step to prove the above result, we show that:
\begin{itemize}
\item{} ~~At least one of the above equalities holds. i.e. we can not find a subsequence $\{i_k\}^{\infty}_{k=1}\subset \mathbb{Z}^+$ which makes both \eqref{eq3} and \eqref{eq4} false.
\end{itemize}

Suppose the claim is not true, i.e., there exists a positive number $A$, and subsequences $\{p_{i_{k}}\}$ and $\{q_{i_{k}}\}$ such that
$$\max\{d(\gamma_{v}(0),p_{i_{k}}), d(\gamma_{v}(0),q_{i_{k}})\}\leq A, ~~k=1,2,3,\cdots.$$
By passing a subsequence if necessary, this implies that $p_{i_{k}}$ converges to some point $\gamma_{v}(-t_{1})$, $t_{1}\in [0,A]$, while $q_{i_{k}}$ converges to some point $\gamma_{v}(t_{2})$, $t_{2}\in [0,A]$.
Let $\epsilon = \frac{c}{8}$, then $\exists\ K_{1} \in \mathbb{Z}^+$ such that for all $k \geq K_{1}$, we have
$$\max\{d(\gamma_{v}(-t_{1}),p_{i_{k}}), d(\gamma_{v}(t_{2}),q_{i_{k}})\}\leq \frac{\epsilon}{2} = \frac{c}{16}.$$
Moreover, by the continuity of the distance function $d(x,y)$, there is a $K_{2} \in \mathbb{Z}^+$ such that for all $k \geq K_{2}$,
$$|d(p_{i_{k}},q_{i_{k}})-d(\gamma_{v}(-t_{1}),\gamma_{v}(t_{2}))|\leq \epsilon =\frac{c}{8}.$$
This implies that when $k>\max\{K_1, K_2\}$, we have
$$d(q_{i_{k}},\gamma_{v}(d(p_{i_{k}},q_{i_{k}})-t_{1}))\leq \frac{c}{4}.$$

We use $\gamma_{i_{k}}$ to denote the geodesic from $p_{i_{k}}$ to $q_{i_{k}}$ satisfying $\gamma_{i_{k}}(-t_{1})=p_{i_{k}}$, then by Lemma \ref{Ka}, we have
$$d(\gamma_{v}(t),\gamma_{i_{k}}(t))\leq \frac{c}{4}+\frac{c}{8}=\frac{3c}{8},~~t \in [-t_{1},d(p_{i_{k}},q_{i_{k}})-t_{1}].$$
Specifically, let $t=0$, we get a contradiction to the assumption  $d(\gamma_{v}(0),\gamma_{p_{k},q_{k}})\geq c$. This proves the claim, i.e. at least one of the equalities \eqref{eq3} and \eqref{eq4} holds.

The second step is to prove that
\begin{itemize}
\item There doesn't exist a subsequence $\{i_k\}^{\infty}_{k=1}$ such that
$$d(\gamma_{v}(0),p_{i_{k}})\leq A, ~~d(\gamma_{v}(0),q_{i_{k}})\rightarrow +\infty,$$
for some $A >0$.
\end{itemize}
This will lead to the equalities \eqref{eq3}.

Obviously, $d(\gamma_{v}(0),q_{i_{k}})\rightarrow +\infty$ and $q_{i_{k}}\in C(\gamma'_{v}(0),\frac{1}{i_{k}})$ imply $$\lim_{k\rightarrow +\infty} q_{i_{k}} = \gamma_{v}(+\infty).$$
Since $\lim_{k\rightarrow +\infty}p_{i_{k}}=\gamma_{v}(-t_{1})$ for some $t_1\in [0,A]$, by choosing a subsequence if necessary,
we get a limit vector $v_{\infty}:=\lim_{k\rightarrow +\infty}\gamma'_{i_{k}}(-t_{1})\in S_{\gamma_{v}(-t_{1})}X$. By Proposition \ref{pro1}, we have
$$\gamma_{v}(+\infty)=\lim_{k\rightarrow +\infty}q_{i_{k}}=\lim_{k\rightarrow +\infty}\gamma_{i_{k}}(t_{i_{k}})=\gamma_{v_{\infty}}(+\infty).$$
Then Lemma \ref{OS} implies that $v_{\infty}=\gamma'_{v}(-t_{1})$, which leads to $\gamma'_{i_{k}}(-t_{1})\rightarrow \gamma'_{v}(-t_{1})$.
So for $\epsilon = \frac{c}{2}$, there is a constant $K_{3} \in \mathbb{Z}^+$ such that $\forall k \geq K_{3}$,
$$d(\gamma_{i_{k}}(t),\gamma_{v}(t))\leq \epsilon = \frac{c}{2}, ~~t \in [-t_{1},t_{1}].$$
Take $t=0$, we have
$$d(\gamma_{i_{k}}(0),\gamma_{v}(0))\leq \frac{c}{2} < c,$$
which also contradicts to the assumption  $d(\gamma_{v}(0),\gamma_{p_{k},q_{k}})\geq c$. We have done with the proof of equality \eqref{eq3}.

The equality \eqref{eq4} can be proved in a similar way, so we omit the proof.  In summary, both of \eqref{eq3} and \eqref{eq4} hold, thus
$$\lim_{k\rightarrow +\infty}p_{k}=\gamma_{v}(-\infty), ~~\lim_{k\rightarrow +\infty}q_{k}=\gamma_{v}(+\infty).$$

In order to prove Proposition \ref{pro4}, we consider the following two cases:

\vspace{1ex}
$\bullet$ \textbf{Case I}~~there exists a positive constant $R$ such that
$$c \leq d(\gamma_{v}(0),\gamma_{k})\leq R, ~~k \in \mathbb{Z}^+.$$
where $\gamma_{k}$ is the geodesic connecting $p_{k}$ and $q_{k}$.

We choose the parametrization of $\gamma_{k}$ such that $d(\gamma_{v}(0),\gamma_{k}(0))=d(\gamma_{v}(0),\gamma_{k})$. Passing to a subsequence if necessary, we can assume $v'=\lim_{k\rightarrow +\infty}\gamma'_{k}(0)$. Then $v'\neq \gamma'_{v}(t)$ for all $t \in \mathbb{R}$, and by Proposition \ref{pro1}, we have
$$\gamma_{v'}(-\infty)=\gamma_{v}(-\infty), ~~\gamma_{v'}(+\infty)=\gamma_{v}(+\infty).$$
This leads to the existence of a flat strip bounded by $\gamma_v$ and $\gamma_{v'}$ (cf.~\cite{Os}). Thus $\text{rank}(v)\geq 2$.

\vspace{1ex}
$\bullet$ \textbf{Case II}~~$d(\gamma_{v}(0),\gamma_{k})$ is unbounded for $k \in \mathbb{Z}^+$.

For each $k\in \mathbb{Z}^+$, choose $p'_{k}, q'_{k}\in \gamma_{v}$ such that
$$d(p_{k},p'_{k})=d(p_{k},\gamma_{v}), ~~d(q_{k},q'_{k})=d(q_{k},\gamma_{v}).$$
Let $b_{k}:[0,1]\rightarrow X$ be a smooth curve connecting $q'_{k}$ and $q_{k}$ with $b_{k}(0)=q'_{k}$ and $b_{k}(1)=q_{k}$, and $\measuredangle_{p}(\gamma_{v}(+\infty),b_{k}(s))$ is an increasing function of $s$, where $p=\pi(v)$. Similarly, let $c_{k}:[0,1]\rightarrow X$ be a smooth curve connecting $p'_{k}$ and $p_{k}$ with $c_{k}(0)=p'_{k}$ and $c_{k}(1)=p_{k}$, and $\measuredangle_{p}(\gamma_{v}(-\infty),c_{k}(s))$ is an increasing function of $s$. Let $\gamma_{k,s}$ be the geodesic connecting $c_{k}(s)$ and $b_{k}(s)$ for each $s \in [0,1]$, with the parametrization $\gamma_{k,0}(0)=\gamma_{v}(0)$ and $\gamma_{k,s}(0)$ is a smooth curve of $s$.

Since $d(p,\gamma_{k}(0))=d(p,\gamma_{k,1}(0))\geq c$ and $d(p,p)=d(p,\gamma_{k,0}(0))=0$, then there exists $s_{k}\in (0,1]$ such that $d(p,\gamma_{k,s_{k}}(0))=c >0$. Without loss of generality we can assume that $\lim_{k\rightarrow +\infty}\gamma'_{k,s_{k}}(0)= v_{\infty}\in SX$. Obviously $d(p,\pi(v_{\infty}))=c>0$.

Using the same argument in the beginning of the proof of this proposition, we have
$$d(b_{k}(s_{k}),p)\rightarrow +\infty, ~~d(c_{k}(s_{k}),p)\rightarrow +\infty.$$
Then by the fact
$$\max\{\measuredangle_{p}(\gamma_{v}(+\infty),b_{k}(s_{k})),\ \measuredangle_{p}(\gamma_{v}(-\infty),c_{k}(s_{k}))\}\leq \frac{1}{k}.$$
we get
$$\lim_{k\rightarrow +\infty}b_{k}(s_{k})=\gamma_{v}(+\infty),~~\lim_{k\rightarrow +\infty}c_{k}(s_{k})=\gamma_{v}(-\infty).$$
By Proposition \ref{pro1}, we know that
$$\gamma_{v_{\infty}}(+\infty)=\lim_{k\rightarrow +\infty}\gamma_{k,s_{k}}(+\infty)=\lim_{k\rightarrow +\infty}b_{k}(s_{k})=\gamma_{v}(+\infty),$$
$$\gamma_{v_{\infty}}(-\infty)=\lim_{k\rightarrow +\infty}\gamma_{k,s_{k}}(-\infty)=\lim_{k\rightarrow +\infty}c_{k}(s_{k})=\gamma_{v}(-\infty).$$
Since $X$ has no focal points, it's impossible that both of $b_{k}(s_{k})$ and $c_{k}(s_{k})$ are on the geodesic $\gamma_{v}$. Then for all $t\in \mathbb{R}$,  $v_{\infty}\neq \gamma'_{v}(t)$. This implies that $\gamma_v$ and $\gamma_{v_{\infty}}$ bound a flat strip, so $\text{rank}(v)\geq 2$. We complete the proof of this proposition.
\end{proof}

The following two propositions are corollaries of Proposition \ref{pro4}.
\begin{proposition}\label{pro5}
Let $X$ be a simply connected manifold without focal points and $\text{rank}(v)=1$,
where $v\in SX$ is a unit vector. Then there is an $\epsilon > 0$ such that, for any $\xi \in C_{\epsilon}(-v)$ and $\eta \in C_{\epsilon}(v)$,
there is a geodesic $\gamma_{\xi,\eta}$ connecting $\xi$ and $\eta$, i.e.,
$\gamma_{\xi,\eta}(-\infty)=\xi$ and $\gamma_{\xi,\eta}(+\infty)=\eta$.
\end{proposition}
\begin{proof}
Since $\text{rank}(v)=1$, by Proposition \ref{pro4}, there exist $k_{0}\in \mathbb{Z}^+$ and a constant number $c > 0$ such that for any points $p\in C(-v,\frac{1}{k_{0}})$ and $q\in C(v,\frac{1}{k_{0}})$, we have $d(\gamma_{v}(0),\gamma_{p,q}) < c$.

Pick up points $p_{i}\in C(-v,\frac{1}{k_{0}})$ and $q_{i}\in C(v,\frac{1}{k_{0}})$ such that $p_{i}\rightarrow \xi$, $q_{i}\rightarrow \eta$. We choose the parametrization of the geodesic $\gamma_{p_{i},q_{i}}$ such that $\gamma_{p_{i},q_{i}}(0)\in B(\gamma_{v}(0),c)$. By the compactness, passing to a subsequence if necessary, we can assume $w = \lim_{i\rightarrow +\infty}\gamma'_{p_{i},q_{i}}(0)$. Then Proposition \ref{pro1} implies that $\gamma_{w}(-\infty)=\xi$ and $\gamma_{w}(+\infty)=\eta$, thus $\gamma_{w}$ is the geodesic $\gamma_{\xi,\eta}$ we want.
\end{proof}

\begin{proposition}\label{pro6}
Let $X$ be a simply connected manifold without focal points and $rank(v)=1$, where $v\in SX$ is a unit vector. Then for any $\epsilon > 0$, there are neighborhoods $U_{\epsilon}$ of $\gamma_{v}(-\infty)$ and $V_{\epsilon}$ of $\gamma_{v}(+\infty)$ such that for each pair $(\xi,\eta)\in U_{\epsilon}\times V_{\varepsilon}$, the geodesic $\gamma_{\xi,\eta}$ defined in Proposition \ref{pro5} is rank $1$ and $d(\gamma_{v}(0),\gamma_{\xi,\eta})<\epsilon$.
\end{proposition}
\begin{proof}
By Proposition \ref{pro5}, for sufficiently small neighborhoods $U_{\epsilon}$, $V_{\epsilon}$ and all $\xi\in U_{\epsilon}$, $\eta\in V_{\epsilon}$,
there is also a geodesic $\gamma_{\xi,\eta}$ connecting $\xi$ and $\eta$.

Now we prove $d(\gamma_{v}(0),\gamma_{\xi,\eta})<\epsilon$. Suppose this is not true, then we can find a positive number $\epsilon_{0}$ and two sequences $\{\xi_{n}\}^{+\infty}_{n=1},\ \{\eta_{n}\}^{+\infty}_{n=1}\subset X(\infty) $ such that for all $n\in\mathbb{Z}^+$,
\begin{enumerate}
\item{} ~~$\lim_{n\rightarrow +\infty}\xi_{n}=\gamma_{v}(-\infty), ~~\lim_{n\rightarrow +\infty} \eta_{n}=\gamma_{v}(+\infty)$,
\item{} ~~The connecting geodesic $\gamma_{\xi_{n},\eta_{n}}$ satisfies $d(\gamma_{v}(0),\gamma_{\xi_{n},\eta_{n}}) \geq \epsilon_{0}$.
\end{enumerate}
Then by a similar argument of Proposition \ref{pro4}, we can show that $\text{rank}(v)\geq 2$, which contradict our assumption that $\text{rank}(v)=1$.

Second, we prove all $\gamma_{\xi,\eta}$'s are rank $1$ geodesics. By contradiction, we assume that, there are two sequences $\{\xi_{n}\}^{+\infty}_{n=1}\subset X(\infty)$ and
$\{\eta_{n}\}^{+\infty}_{n=1}\subset X(\infty)$ such that,
$$\lim_{n\rightarrow +\infty}\xi_{n}=\gamma_{v}(-\infty), ~~\lim_{n\rightarrow +\infty} \eta_{n}=\gamma_{v}(+\infty),$$
and the connecting geodesic $\gamma_n:=\gamma_{\xi_{n},\eta_{n}}$ is not rank $1$.

We choose the parametrization of $\gamma_{n}$ such that $d(\gamma_{v}(0),\gamma_{n}(0))=d(\gamma_{v}(0),\gamma_{n})$.

Since $\{\gamma_{n}(0)\}$ are in a compact subset, passing to a subsequence if necessary,
we consider the following two cases:\\
\textbf{Case I}~~$\lim_{n\rightarrow +\infty}\gamma_{n}(0)=\gamma_{v}(0)$. Since
$$\gamma_{n}(-\infty)=\xi_{n} \rightarrow \gamma_{v}(-\infty), ~~\gamma_{n}(+\infty)=\eta_{n} \rightarrow \gamma_{v}(+\infty),$$
by Lemma \ref{OS} and Proposition \ref{pro1}, we have $\lim_{n\rightarrow +\infty}\gamma'_{n}(0)=\gamma'_{v}(0)=v$.
Since higher rank set $\text{sing}$ is closed, $\text{rank}(v)\geq 2$, a contradiction.\\
\textbf{Case II}~~$\lim_{n\rightarrow +\infty}\gamma_{n}(0)\neq\gamma_{v}(0)$.
In this case, if we denote by $v'=\lim_{n\rightarrow +\infty}\gamma'_{n}(0)$, then $v'\neq v$.
By Proposition \ref{pro1},
$$\gamma_{v'}(-\infty)=\lim_{n\rightarrow +\infty}\gamma_{n}(-\infty)=\gamma_{v}(-\infty),
~~\gamma_{v'}(+\infty)=\lim_{n\rightarrow +\infty}\gamma_{n}(+\infty)=\gamma_{v}(+\infty).$$
Thus $\text{rank}(v)\geq 2$, a contradiction.
\end{proof}

An isometry $\alpha \in \text{Iso}(X)$ is called \emph{axial} if there exists a geodesic $\gamma$ and a $t_{0}>0$ such that for any $t\in \mathbb{R}$, $\alpha(\gamma(t)) = \gamma(t+t_{0})$. Correspondingly $\gamma$ is called an \emph{axis} of $\alpha$.
Axial isometry is abundant for compact manifolds without conjugate points.
Each element in $\Gamma$ is axial if $X/\Gamma$ is compact (cf. \cite{CS} Lemma 2.1).
The following property of axial isometry was proved by Watkins:


\begin{lemma}[cf. Watkins \cite{Wat}]\label{pro7}
Let $X$ be a simply connected manifold without focal points
and let $\gamma$ be a rank $1$ geodesic axis of $\alpha \in \Gamma \subset \text{Iso}(X)$.
For all neighborhood $U \subset \overline{X}$ of $\gamma(-\infty)$ and $V \subset \overline{X}$ of $\gamma(+\infty)$,
there is $N \in \mathbb{Z}^+$ such that
$$\alpha^{n}(\overline{X}-U)\subset V, ~~\alpha^{-n}(\overline{X}-V)\subset U$$
for all $n \geq N$.
\end{lemma}

Lemma \ref{pro7} and Proposition \ref{pro6} imply the following results.
\begin{corollary}\label{co2.11}
Let $X$ be a simply connected manifold without focal points, then
an end point of a rank $1$ axis can be connected by rank $1$ geodesics to all the other points in $X(\infty)$.
This implies that, for every $\xi\in X(\infty)$, there is a rank $1$ geodesic $\gamma_+$ with $\gamma_+(+\infty)=\xi$ and a rank $1$ geodesic $\gamma_{-}$ with $\gamma_{-}(-\infty)=\xi$.
\end{corollary}
\begin{proof}
Let $\gamma$ be a rank $1$ geodesic axis of $\alpha \in \Gamma \subset \text{Iso}(X)$.
By Proposition \ref{pro6} there exist neighborhoods $U$ of $\gamma(-\infty)$ and $V$ of $\gamma(+\infty)$ in $X(\infty)$, such that for each pair
$(\xi,\eta)\in U \times V$, there is a rank 1 geodesics connecting
$\gamma_{+}=\gamma_{\gamma(+\infty),\xi}$ and $\gamma_{-}=\gamma_{\eta,\gamma(-\infty)}$.

For each $\xi \in X(\infty)$ with $\xi \neq \gamma(+\infty)$, Lemma \ref{pro7} implies that there exists $N>0$ such that $\alpha^{-N}\xi \in U$.
By Proposition \ref{pro6} there exists a rank $1$ connecting geodesic $\gamma_{\gamma(+\infty),\alpha^{-N}\xi}$.
Notice that $\gamma$ is a rank $1$ geodesic axis of $\alpha$, and $\gamma_{+}:=\alpha^{N}\gamma_{\gamma(+\infty),\alpha^{-N}\xi}$ is a geodesic connecting
$\alpha^N(\gamma(+\infty))=\gamma(+\infty)$ and $\xi$.
The fact that $\gamma_{\gamma(+\infty),\alpha^{-N}\xi}$ is rank 1 implies that $\gamma_{+}$ is a rank 1 geodesic connecting $\gamma(+\infty)$ and $\xi$.

Similarly, we can show that for any $\eta \in X(\infty)$ and $\eta \neq \gamma(-\infty)$, there exists rank 1 geodesic connecting them.
\end{proof}

The connecting geodesic in Corollary \ref{co2.11} is unique since it is rank 1.

Combining Proposition \ref{pro3} and Lemma \ref{pro7}, we can get the following result. We remark that this result is the no focal points version of Lemma $2.3$ and Corollary $2.4$ of \cite{Kn2}.

\begin{proposition}\label{pro8}
Let $X$ be a simply connected manifold without focal points and let
$\Gamma \subset \text{Iso}(X)$ be
such that $M=X/\Gamma$ is compact. If $\gamma$ is a rank $1$ geodesic on X,
then there exists a sequence of rank $1$ axes $\gamma_{n}$ of deck transformations
$\alpha_{n}\in \Gamma$
such that $\gamma_{n} \rightarrow \gamma$. Furthermore, $\Gamma$ acts minimally on $X(\infty)$,
i.e., for any $\xi \in X(\infty)$, $\overline{\Gamma\xi}=X(\infty)$.
\end{proposition}

\begin{proof}
By Proposition \ref{pro3}, there is a sequence $\{\alpha_{n}\}^{+\infty}_{n=1}\subset \Gamma$  such that
\begin{equation}\label{eq2.11}
\alpha^{-1}_{n}(p)\rightarrow \gamma(+\infty),~\mbox{and}~~\alpha_{n}(p)\rightarrow \gamma(-\infty),\ \forall\ p\in X.
\end{equation}
Lemma \ref{pro7} implies that the closure of sufficiently small neighborhoods $U_{n}$ of $\gamma(-\infty)$ and
$V_{n}$ of $\gamma(+\infty)$ in $\overline{X}$ with
$$\bigcap_{n=1}^{\infty}U_{n}=\{\gamma(-\infty)\},~~\bigcap_{n=1}^{\infty}V_{n}=\{\gamma(+\infty)\}$$
are invariant under $\alpha_{n}$ and $\alpha^{-1}_{n}$ respectively.
Since in cone topology, $\overline{X}=X \cup X(\infty)$ is homeomorphic to the $\dim(X)$-unit ball $B_{\dim(X)}$ in $\mathbb{R}^{\dim(X)}$,
we know $\overline{U}_{n}$ and $\overline{V}_{n}$ are both homeomorphic to $B_{\dim(X)}$.
Then by the Brouwer fixed point theorem,
$\alpha_{n}$ has one fixed point in $p_{n}\in U_{n}$ and $\alpha^{-1}_{n}$ has one fixed point in $q_{n}\in V_{n}$.

We claim that both $p_{n}$ and $q_{n}$ belong to $X(\infty)$ for sufficiently large $n$. We prove this fact by contradiction.
First, if $p_{n}\in X$ and $q_{n}\in X$, then $\alpha_{n}=\mbox{id}$ since $\alpha_{n} \in \Gamma \subset \mbox{Iso}(X)$, which contradicts to \eqref{eq2.11}.
Second, if $p_{n}\in X$ and $\xi=q_{n}\in X(\infty)$, then $\alpha_{n}$ maps the geodesic ray $\gamma_{p_{n},\xi}$ to itself and fixes endpoints,
thus $\alpha_{n}=\mbox{id}$, which also contradicts to \eqref{eq2.11}. The case $p_{n}\in X(\infty)$ and $q_{n}\in X$ can be considered in the same way.
Thus both $p_{n}$ and $q_{n}$ belong to $X(\infty)$ for large enough $n$.

Without loss of generality, we assume that both $p_{n}$ and $q_{n}$ belong to $X(\infty)$ for all $n\geq 1$. By the definition of $U_{n}$ and $V_{n}$, Proposition \ref{pro6} implies that for large enouth $n$, the connecting geodesic $\gamma_{n}=\gamma_{p_{n},q_{n}}$
is the rank $1$ axis of the axial isometry $\alpha_{n}$. Furthermore,
$$\gamma_{n}(-\infty)=p_{n}\rightarrow \gamma(-\infty), ~~  \gamma_{n}(+\infty)=q_{n}\rightarrow \gamma(+\infty),$$
thus $\gamma_{n}\rightarrow \gamma$.

Now we will prove $\overline{\Gamma\xi}=X(\infty)$. We only need to show that for any $\eta \in X(\infty)$ with $\eta \neq \xi$,
there exists a sequence $\{\beta_{n} \}\subset \Gamma$ such that $\beta_{n}(\xi)\rightarrow \eta$.
It's easy to see that the rank function
$$R: SM\rightarrow \mathbb{Z}^+,~~v \mapsto R(v)=\text{rank} (v)$$
is upper semi-continuous and integer-valued, thus the set of vectors $v\in SM$ such that $\text{rank} (w)=\text{rank} (v)$ for all $w$ sufficiently close to $v$ is dense in $SM$. Since this set is exactly the set of rank $1$ vectors (cf. \cite{Wat}), it follows that the set of rank $1$ vectors is dense in $SM$. Therefore the set of rank $1$ vectors is also dense in $SX$.
The above result implies that the set of vectors tangent to rank $1$ axes are dense in SX.
Therefore, for any open neighborhood $U \subset X(\infty)$ of $\eta$, there is a rank $1$ axis $\gamma_{v}$ of some $\alpha \in \Gamma$,
satisfying that $\gamma_{v}(+\infty) \in U$  and $\gamma_{v}(-\infty) \neq \xi$. Then by Lemma \ref{pro7} there exists $n \in \mathbb{Z}^+$ such that
$\alpha^{n}\xi \in U$. Since $U$ can be chosen arbitrarily small,
we can find a sequence $\{n_{k}\}\subset \mathbb{Z}^+$ and $\{\alpha_{k}\}\subset \Gamma$ such that
$\lim_{k\rightarrow +\infty}\alpha_{k}^{n_{k}}\xi=\eta$.
\end{proof}

\section{\textbf{Existence and uniqueness of the Busemann density}}\label{Busemann}
\subsection{Construction of a Busemann density}\label{Busemann1}
For each pair of points $(p,q)\in X \times X$ and each point at infinity $\xi \in X(\infty)$, the \emph{Busemann function} determined by $p, q$ and $\xi$ is
$$b_{p}(q,\xi):=\lim_{t\rightarrow +\infty}\{d(q,\gamma_{p,\xi}(t))-t\},$$
where $\gamma_{p,\xi}$ is the geodesic ray from $p$ to $\xi$.
The Busemann function $b_{p}(q,\xi)$ is well-defined since the function $t \mapsto d(q,\gamma_{p,\xi}(t))-t$
is bounded from above by $d(p,q)$, and decreasing in $t$ (this can be checked by using the triangle inequality).

The level sets of the Busemann function $b_{p}(q,\xi)$ are called the horospheres centered at $\xi$.
We denote by $H_{\xi}(\gamma_{p,\xi}(t))$ the horosphere centered at $\xi$ and passing through $\gamma_{p,\xi}(t)$.
For more details of the Busemann functions and horosphpers, please see ~\cite{DPS,Ru1,Ru2}.
At here, we restate one important property of the horospheres, which will be used later.

\begin{theorem}[cf.~Ruggiero \cite{Ru1}]\label{th3.15}
Horospheres are continuous with their centers, i.e., $H_{\xi}(\gamma_{p,\xi}(t))$ depends continuous on $\xi$.
\end{theorem}

Remark: It is straightforwardly that $H_{\xi}(\gamma_{p,\xi}(t))$ depends continuously on $t$.

\begin{Defi}\label{def3}
Let $X$ be a simply connected Riemannian manifold without conjugate points and
$\Gamma \subset \text{Iso}(X)$ be a discrete subgroup. For a given constant $r>0$, a family of finite Borel measures $\{\mu_{p}\}_{p\in X}$
on $X(\infty)$ is called an $r$-dimensional Busemann density if
\begin{enumerate}
\item For any $p,\ q \in X$ and $\mu_{p}$-a.e. $\xi\in X(\infty)$,
$$\frac{d\mu_{q}}{d \mu_{p}}(\xi)=e^{-r \cdot b_{p}(q,\xi)}$$
 where $b_{p}(q,\xi)$ is the Busemann function.
\item $\{\mu_{p}\}_{p\in X}$ is $\Gamma$-equivariant, i.e., for all Borel sets $A \subset X(\infty)$ and for any $\alpha \in \Gamma$,
we have $$\mu_{\alpha p}(\alpha A) = \mu_{p}(A).$$
\end{enumerate}
\end{Defi}

The construction of such a Busemann density is due to Patterson \cite{Pat} in the case of Fuchsian groups, and generalized by Sullivan \cite{Su} for hyperbolic spaces.

Let $X$ be a simply connected Riemannian manifold without conjugate points and $\Gamma \subset \text{Iso}(X)$ be a infinite discrete subgroup. For each pair of points $(p,p_{0})\in X \times X$ and $s \in \mathbb{R}$,
\emph{Poincar\'e series} is defined as
\begin{equation}\label{poincare series}
P(s,p,p_{0}):= \sum_{\alpha \in \Gamma}e^{-s\cdot d(p,\alpha p_{0})}.
\end{equation}
This series may diverge for some $s>0$. The number $\delta:= \inf \{s \in \mathbb{R}\mid P(s,p,p_{0}) < +\infty\}$ is called the \emph{critical exponent of $\Gamma$}. By the triangle inequality, it is easy to check that $\delta$ is independent of $p$ and $p_{0}$. We say $\Gamma$ is of \emph{divergent type} if the Poincar\'e series diverges when $s=\delta$.
It is well-known that whether $\Gamma$ is of divergent type or not is independent of the choices of $p$ and $p_{0}$.
The following lemma shows the relation between the critical exponent and the topological entropy of the geodesic flow on $M=X/\Gamma$.

\begin{lemma}\label{pro9}
Let $(M,g)$ be a compact Riemannian manifold without conjugate points and $X$ be its universal cover, then $\delta=h_{\text{top}}(g)$, where $h_{\text{top}}(g)$ is the topological entropy of geodesic flow on $SM$.
\end{lemma}
\begin{proof}
Since $X$ is the universal cover of $M$, by Freire and Ma\~n\'e's Theorem (cf.~\cite{FrMa}),
we know that $$h(g)=h_{\text{top}}(g),$$
where
$$h(g):=\lim_{r\rightarrow +\infty}\frac{\log \text{Vol}B(p,r)}{r}$$
is the volume entropy of the geodesic flow. Here $B(p,r)$ denotes the ball in $X$ centered at $p \in X$ with radius $r>0$.
In ~\cite{Man}, A. Manning showed that this limit is independent of $p$.

For each pair $(p,p_{0})\in X \times X$  and $k\geq 0$, we define
$$S_{k}:= \sharp\{\Gamma p_{0} \cap (B(p,k+1)-B(p,k))\}.$$
Then
$$\sum^{+\infty}_{k=0}S_{k}\cdot e^{-s(k+1)}\leq P(s,p,p_{0}) \leq\sum^{+\infty}_{k=0}S_{k}\cdot e^{-sk}
= \sum^{+\infty}_{k=0}e^{(\frac{\ln S_{k}}{k}-s)k}.$$
Therefore, $$\delta = \varlimsup_{k \rightarrow +\infty}\frac{1}{k}\ln S_{k} .$$

By the definition of $h(g)$, for any $\epsilon>0$, there exists $K_0\in \mathbb{Z}^+$ such that if $r\geq K_0$, then
$$e^{r(h(g)-\epsilon)}\leq \text{Vol}B(p,r)\leq e^{r(h(g)+\epsilon)}.$$
Let $d$ and $A$ be the diameter and area of a fundamental domain $\mathcal{F}$ respectively. Then there exists $K_1\in \mathbb{Z}^+$ such that for any $k\geq \max\{K_0,K_1\}$,
\begin{equation*}
\begin{aligned}
S_{k} \leq \frac{1}{A}(\text{Vol}B(p,k+1+d)-\text{Vol}B(p,k-d))\leq \frac{1}{A}\text{Vol}B(p,k+1+d)\leq \frac{1}{A}e^{(k+1+d)(h(g)+\epsilon)}.
\end{aligned}
\end{equation*}
Then it follows that
\begin{equation}\label{e:half1}
\begin{aligned}
\delta = \varlimsup_{k \rightarrow +\infty}\frac{1}{k}\ln S_{k}\leq h(g)+\epsilon.
\end{aligned}
\end{equation}

On the other hand, we claim that there exists $k_i\to \infty$ such that
\[
\text{Vol}B(p,k_i+1)-\text{Vol}B(p,k_i)\geq e^{(h(g)-\epsilon)k_i}.
\]
Assume not. Then there exists $K_2\in \mathbb{Z}^+$, such that for any $k>K_2$,
\begin{eqnarray*}
\text{Vol}B(p,k+1)-\text{Vol}B(p,k)&<&e^{(h(g)-\epsilon)k},\\
\text{Vol}B(p,k+2)-\text{Vol}B(p,k+1)&<&e^{(h(g)-\epsilon)(k+1)},\\
&\vdots&\\
\text{Vol}B(p,k+N)-\text{Vol}B(p,k+N-1)&<&e^{(h(g)-\epsilon)(k+N-1)}.
\end{eqnarray*}
This implies that
\[
\text{Vol}B(p,k+N)<(e^{(h(g)-\epsilon)k}+\cdots+e^{(h(g)-\epsilon)(k+N-1)})+\text{Vol}B(p,k)
\]
\begin{equation*}
<Be^{(h(g)-\epsilon)(k+N-1)}+e^{(h(g)+\epsilon)k}.
\end{equation*}
where $B=1+e^{-(h(g)-\epsilon)}+\cdots+e^{-(h(g)-\epsilon)N}+\cdots <+\infty$. It follows that
\[
\texttt{Vol}B(p,k+N)<2Be^{(h(g)-\epsilon)(k+N-1)}.
\]
Setting $N\rightarrow\infty$, then $\lim\limits_{k\rightarrow\infty}\frac{1}{k}\log \text{Vol}B(p,k)\leq h(g)-\epsilon$. We get a contradiction.


Now observe that
\begin{equation*}
\begin{aligned}
\int_{\mathcal{F}}\sum_{\gamma\in \Gamma} e^{-sd(p, \gamma y)}d\text{Vol}(y)&=\int_X e^{-sd(p, y)}d\text{Vol}(y)
&\geq \sum_{k=0}^\infty e^{-s(k+1)}(\text{Vol}B(p,k+1)-\text{Vol}B(p,k)).
\end{aligned}
\end{equation*}

If we let $s= h(g)-2\epsilon$ and use the claim above, then the right side of the above inequality is infinity. Thus
\begin{equation}\label{e:half2}
\delta\geq h(g)-2\epsilon.
\end{equation}
So $\delta=h(g)=h_{\text{top}}(g)$ by \eqref{e:half1} and \eqref{e:half2}.
\end{proof}

Fix a point $p_{0}\in X$. For each $s > \delta = h$ where $h=h_{\text{top}}(g)$ and $p \in X$, consider the measure:
$$\mu_{p,s}:=\frac{\sum_{\alpha \in \Gamma}e^{-s\cdot d(p,\alpha p_0)}\delta_{\alpha p_{0}}}{\sum_{\alpha\in\Gamma}e^{-s \cdot d(p_{0},\alpha p_{0})}},$$
where $\delta_{\alpha p_{0}}$ is the Dirac measure at point $\alpha p_{0}$. The following properties of $\mu_{p,s}$ can be easily checked by using the triangle inequality:
\begin{enumerate}
\item $e^{-s\cdot d(p, p_{0})} \leq \mu_{p,s}(\overline{X}) \leq e^{s\cdot d(p, p_{0})}$.
\item $\Gamma p_{0} \subset \text{supp}\mu_{p,s} \subset \overline{\Gamma p_{0}}$, where $\Gamma p_{0}$ is the orbit of $p_0\in X$ under the action of $\Gamma$.
\end{enumerate}

Here in constructing a Busemann density,  we consider the case that the group $\Gamma$ is of divergent type first.
At the end of this subsection, we give a sketch of Patterson's  method to deal with the convergent type (cf.~\cite{Pat}).

Now for $p\in X$, consider a $weak^{\star}$ limit
$$\lim_{s_{k}\downarrow h}\mu_{p,s_{k}}=\mu_{p}.$$
Since $P(s,p_{0},p_{0})$ is divergent for $s=h$ and $\Gamma$ is discrete, it's obvious that $\text{supp}(\mu_{p})\subset \overline{\Gamma p_{0}} \cap X(\infty)$. In fact, one can easily check that $\lim_{s_{k}\downarrow h}\mu_{p,s_{k}}(A)=0$ for all bounded open sets $A\subset X$, therefore $\text{supp}(\mu_{p})\subset X(\infty)$.
Moreover, we can show that $\{\mu_{p}\}_{p \in X}$ is an $h$-dimensional Busemann density.
We should emphasis that the different choices of $\{s_{k}\}$ may lead different $weak^{\star}$ limits. But in the case of rank 1 manifolds without focal points, we will show that the the Busemann density is unique, i.e., independent of $\{s_{k}\}$.

In order to show $\{\mu_{p}\}_{p \in X}$ constructed above is an Busemann density, we need the following two results about Busemann function.

\begin{proposition}\label{pro3.18}
For each pair of points $p,q\in X$, the map
$$\beta_{p,q}:X(\infty)\rightarrow \mathbb{R},~~\xi\mapsto \beta_{p,q}(\xi):=b_{p}(q,\xi)$$
is continuous.
\end{proposition}
\begin{proof} We know that if $t=b_{p}(q,\xi)$ for some $t\in \mathbb{R}$, then the horosphere passing through $q$ and centered at $\xi$ will intersect the geodesic
$\gamma_{p,\xi}$ perpendicularly at $\gamma_{p,\xi}(-t)$ (cf.~\cite{Ru2} Lemma 4.2).

Proposition \ref{pro2.4} implies that for any $p\in X$, and any $\{\xi_{n}\}\subset X(\infty)$ with $\lim_{n\rightarrow +\infty}\xi_{n}=\xi$,
we have $\lim_{n\rightarrow +\infty}\gamma'_{p,\xi_{n}}(0)=\gamma'_{\xi}(0)$. Therefore, if we fix a point $p \in X$,
then the geodesic $\gamma_{p,\xi}$ depends continuously on $\xi$,
Moreover, Theorem \ref{th3.15} tells us that, for the manifold without focal points, the horospheres depend continuously on their centers. This implies that the intersection points of the geodesic $\gamma_{p,\xi}$ and the horosphere $H_{\xi}(q)$ depends continuously on $\xi\in X(\infty)$.
Then by our explanation in the last paragraph, we know the Busemann function $b_{p}(q,\xi)$ depends continuously on $\xi$.
\end{proof}

\begin{proposition}\label{pro3.19}
For each pair of points $p,q\in X$, if there is a sequence $\{x_{n}\}\subset X$ with $\lim_{n\rightarrow +\infty}x_{n}=\xi$, then $\lim_{n\rightarrow +\infty}\{d(q,x_{n})-d(p,x_{n})\}=b_{p}(q,\xi)$.
\end{proposition}
\begin{proof}
Since $\lim_{n\rightarrow +\infty}x_{n}=\xi$, then by passing to a subsequence if necessary,
we can assume that $x_{n} \in TC(v,\frac{1}{n},n)$,
where $v=\gamma'_{p,\xi}(0)$ and $TC(v,\frac{1}{n},n)$ is the truncated cone introduced in Section \ref{geometric}.
Let $v_{n}=\gamma'_{p,x_{n}}(0)$, then we know that the angle between $v$ and $v_{n}$ is smaller than $\frac{1}{n}$.
This implies $\lim_{n\rightarrow +\infty}v_{n}=v$. Let $\xi_{n}=\gamma_{v_{n}}(+\infty)$,
by Proposition \ref{pro2.4}, $\lim_{n\rightarrow +\infty}\xi_{n}=\xi$.

Fix a point $\xi \in X(\infty)$, let $a=b_{p}(q,\xi)$. By Proposition \ref{pro3.18},
$\lim_{n\rightarrow +\infty}\xi_{n}=\xi$ implies that for any $\epsilon > 0$, there exists $N \in \mathbb{Z}^{+}$,
such that for all $n > N$, we have
$$\mid b_{p}(q,\xi_{n})- b_{p}(q,\xi)\mid < \frac{\epsilon}{3}.$$

Since the function $t\mapsto d(q,\gamma_{w}(t))-t$ is nonincreasing, we know that
for any $\epsilon > 0$, there exists a constant $T \in \mathbb{R}$, such that
$$\mid (d(q,\gamma_{v}(t))-t)- b_{p}(q,\xi)\mid < \frac{\epsilon}{3},~~\forall~t>T.$$
Now we fix the number $T$. Since $\lim_{n\rightarrow +\infty}v_{n}=v$,
we know that there exists $N_{1} \geq \max \{T,N\}$, such that
$$\mid d(\gamma_{v_{n}}(t), \gamma_{v}(t))\mid < \frac{\epsilon}{3}, ~~\forall~t \in [0,T]~\mbox{and}~n>N_1.$$
Thus $\forall n > N_{1}$, we have
$$ \mid (d(q,\gamma_{v_{n}}(T))-d(p,\gamma_{v_{n}}(T)) )  -  (d(q,\gamma_{v}(T))-d(p,\gamma_{v}(T)))  \mid
\leq 2\mid d(\gamma_{v_{n}}(T),\gamma_{v}(T))\mid \leq  \frac{2\epsilon}{3}.$$

\vspace{-1ex}
By the nonincreasing property of the function $t\mapsto d(q,\gamma_{w}(t))-t$, we know that for all $n \geq N_{1}$,
\begin{eqnarray*}
\mid (d(q,x_{n})-d(p,x_{n}))-b_{p}(q,\xi_{n})\mid
&   =  & \mid (d(q,\gamma_{v_{n}}(d(p,x_{n}))-d(p,\gamma_{v_{n}}(d(p,x_{n})))-b_{p}(q,\xi_{n})\mid\\
& \leq & \mid (d(q,\gamma_{v_{n}}(T)-d(p,\gamma_{v_{n}}(T))-b_{p}(q,\xi_{n})\mid.
\end{eqnarray*}
Thus
\begin{eqnarray*}
\mid (d(q,x_{n})-d(p,x_{n}))-b_{p}(q,\xi)\mid
& \leq & \mid (d(q,\gamma_{v_{n}}(T))-d(p,\gamma_{v_{n}}(T)))-(d(q,\gamma_{v}(T))-d(p,\gamma_{v}(T)))\mid \\
& + & \mid (d(q,\gamma_{v}(T)-d(p,\gamma_{v}(T))-b_{p}(q,\xi_{n})\mid\\
& \leq & \frac{2\epsilon}{3}+\frac{\epsilon}{3}=\epsilon.
\end{eqnarray*}
We are done with the proof.
\end{proof}

\begin{proposition}\label{pro10}
$\{\mu_{p}\}_{p \in X}$ is an h-dimensional Busemann density where $h=h_{\text{top}}(g)$.
\end{proposition}
\begin{proof}
We shall prove that $\{\mu_{p}\}_{p \in X}$ satisfies the two properties in the Definition \ref{def3}.

By Proposition \ref{pro3}, for any $\xi\in X(\infty)$, there exists a sequence $\{\alpha_{k}\}\subset \Gamma$ such that for any point $p_{0}\in X$,
$\xi = \lim_{k\rightarrow +\infty}\alpha_{k}p_{0}$.
For each $p,q \in X$, we compute the ratio of the coefficient of $\delta_{\alpha_{k}p_{0}}$ of measures $\mu_{q,s}$ and $\mu_{p,s}$,
by the Proposition \ref{pro3.19}
$$\frac{e^{-s\cdot d(q,\alpha_{k}p_{0})}}{e^{-s\cdot d(p,\alpha_{k}p_{0})}}=
e^{-s\cdot (d(q,\alpha_{k}p_{0})-d(p,\alpha_{k}p_{0}))}\rightarrow e^{-s\cdot b_{p}(q,\xi)},\ \text{as}\ k\rightarrow +\infty.$$
Since the discrete subgroup $\Gamma$ satisfying that $X/\Gamma$ is compact, we know
2a:=$\inf\{d(x,\alpha x)\mid x \in X, \alpha \in \Gamma\} > 0.$
For each $k \in \mathbb{Z}^{+}$, we denote by $B(k)$ the ball centered at $\alpha_{k}p_{0}$ with radius $a$,
then for each $s > h$,
$$\frac{\mu_{q,s}}{\mu_{p,s}}(B(k))=\frac{e^{-s\cdot d(q,\alpha_{k}p_{0})}}{e^{-s\cdot d(p,\alpha_{k}p_{0})}}.$$
This implies that
$$\frac{d\mu_{q}}{d \mu_{p}}(\xi)=\lim_{k\rightarrow +\infty}\frac{\mu_{q,s_{k}}}{\mu_{p,s_{k}}}(B(k))=e^{-h \cdot b_{p}(q,\xi)}.$$

In fact, for any $s > h$, we have
\begin{eqnarray*}
\mu_{\alpha p,s}
&=& \frac{\sum_{\beta \in \Gamma}e^{-s\cdot d(\alpha p,\beta p_{0})}\delta_{\beta p_{0}}}{P(s,p_{0},p_{0})}
=\frac{\sum_{\beta \in \Gamma}e^{-s\cdot d(\alpha p,\alpha\beta p_{0})}\delta_{\alpha\beta p_{0}}}{P(s,p_{0},p_{0})}\\
&=& \frac{\sum_{\beta \in \Gamma}e^{-s\cdot d(p,\beta p_{0})}\delta_{\alpha\beta p_{0}}}{P(s,p_{0},p_{0})}.
\end{eqnarray*}
So for any Borel measurable subset $B \subset \overline{X}=X \cup X(\infty)$, we have $\mu_{\alpha p, s}(\alpha B) = \mu_{p,s}(B)$. Therefore, for any Borel measurable subset $A \subset X(\infty)$,
$\mu_{p}(A) = \mu_{\alpha p}(\alpha A)$.
\end{proof}
We call a Busemann density constructed as above a \emph{Patterson-Sullivan measure}.

\begin{proposition}\label{pro11}
The total mass of the measures $\{\mu_{p}\}_{p \in X}$ is uniformly bounded from above and below.
\end{proposition}
\begin{proof}
Take a compact fundamental domain $\mathcal{F}$ containing $p_{0}$. Let $K=\text{diam}(\mathcal{F}) < +\infty$.
Since $$e^{-s\cdot d(p, p_{0})} \leq \mu_{p,s}(\overline{X}) \leq e^{s\cdot d(p, p_{0})},$$
for any point $p\in\mathcal{F}$ and $s\in (h, h+1]$, then
$$e^{-(h+1)K}\leq \mu_{p,s}(\overline{X}) \leq e^{(h+1)K}.$$
Thus for any point $p\in\mathcal{F}$,
$$e^{-(h+1)K}\leq \mu_{p}(X(\infty)) \leq e^{(h+1)K}.$$

Now for any point $p \in X$, there exists $\alpha \in \Gamma$ such that $\alpha p \in \mathcal{F}$. So we have
\begin{itemize}
\item{} ~~$\mu_{p}(X(\infty))=\mu_{\alpha p}(\alpha X(\infty))\leq e^{(h+1)K}$;
\item{} ~~$e^{-(h+1)K}\leq \mu_{\alpha p}(X(\infty)) = \mu_{p}(\alpha^{-1} X(\infty))= \mu_{p}(X(\infty))$.
\end{itemize}
Therefore $$e^{-(h+1)K}\leq \mu_{p}(X(\infty))\leq e^{(h+1)K},$$ which means that $\mu_{p}(X(\infty))$ is uniformly bounded from above and below.
\end{proof}

The next proposition shows that a Busemann density has full support equal to $X(\infty)$.
\begin{proposition}\label{pro12}
For any $p \in X$, $\emph{\text{supp}}(\mu_{p})=X(\infty)$.
\end{proposition}
\begin{proof}
Suppose $\text{supp}(\mu_{p})\neq X(\infty)$, then there exists $\xi \in X(\infty)$ which is not in $\text{supp}(\mu_{p})$. Then we can find an open neighborhood $U$ of $\xi$ in $X(\infty)$ such that $\mu_{p}(U)=0$.
Note that Proposition \ref{pro8} shows that $\bigcup_{\alpha \in \Gamma}\alpha(U)=X(\infty)$,
so for any $\eta \in X(\infty)$, there is an element $\alpha \in \Gamma$ such that $\eta \in \alpha U$. By the $\Gamma$-invariance $\mu_{\alpha p}(\alpha U)=\mu_{p}(U)=0$. Since $\mu_{p}$ is equivalent to $\mu_{\alpha p}$, we have $\mu_{p}(\alpha U)=0$. Therefore $\text{supp}(\mu_{p}) = \emptyset$, which contradicts to Proposition \ref{pro11}.
\end{proof}

When $\Gamma$ is of convergent type, Patterson showed in \cite{Pat} that one can weight the Poincar\'e series with a positive monotone increasing function to make the refined Poincar\'e series diverge at $s=h$.
More precisely, we consider a positive monotonely increasing function $g:\mathbb{R}^{+} \rightarrow \mathbb{R}^{+}$ with
$\lim_{t\rightarrow +\infty}\frac{g(x+t)}{g(t)}=1$, $\forall x \in \mathbb{R}^{+}$,
such that the weighted Poincar\'e series $\widetilde{P}(s,p,p_{0}):=\sum_{\alpha\in \Gamma}g(d(p,\alpha p_{0}))e^{-s\cdot d(p,\alpha p_{0})}$
diverges at $s=\delta$.  For the construction of the function $g$, see \cite{Pat} Lemma 3.1. Then, for the weighted Poincar\'e series $\widetilde{P}(s,p,p_{0})$,
we can always construct a Busemann density by using the same method as above.

\subsection{Projections of geodesic balls to $X(\infty)$}\label{proj}

In this subsection, we will show some semi-local properties of a Busemann density $\{\mu_{p}\}_{p\in X}$. These properties will play key roles in our proof of main theorems in subsequent sections.

For each $\xi \in X(\infty)$ and $p \in X$, define the projections
\begin{itemize}
\item{} ~~$pr_{\xi}: X \rightarrow X(\infty), ~~q \mapsto \gamma_{\xi,q}(+\infty)$,
\item{} ~~$pr_{p}: X\backslash \{p\} \rightarrow X(\infty), ~~q \mapsto \gamma_{p,q}(+\infty)$,
\end{itemize}
where $\gamma_{\xi,q}$ is the geodesic satisfying $\gamma_{\xi,q}(-\infty)=\xi$ and $\gamma_{\xi,q}(0)=q$.

For $A \subset X$ and $p \in X$,
$$pr_{p}(A)=\{pr_{p}(q)\mid q \in A \backslash \{p\} \}.$$

\begin{proposition}\label{pro3.2}
The projection map
$$pr: \overline{X}\times X \setminus \mbox{diag} \rightarrow X(\infty),$$
$$(x,p)\mapsto pr(x,p):=pr_{x}(p)=\gamma_{x,p}(+\infty),$$
is continuous.
\end{proposition}
\begin{proof}
First we prove that for each $p \in X$, the map $pr(\cdot,p): \overline{X}\setminus p \rightarrow X(\infty)$ is continuous.
Suppose $\{x_{n}\}\subset \overline{X}$ is a convergent sequence and $x=\lim_{n\rightarrow +\infty}x_{n}\in\overline{X}$.
Corollary \ref{cor2.5} implies that $\lim_{n\rightarrow +\infty}-\gamma'_{p,x_{n}}(0)=-\gamma'_{p,x}(0)$.
Then by Proposition \ref{pro1},
we get $\lim_{n\rightarrow +\infty}(\gamma_{-\gamma'_{p,x_{n}}(0)}(+\infty)=\gamma_{-\gamma'_{p,x}(0)}(+\infty)$,
i.e., $\lim_{n\rightarrow +\infty} pr(x_{n},p)=pr(x,p)$.

Similarly, we can prove that for each $x \in \overline{X}$, the map $pr(x,\cdot): X\setminus x \rightarrow X(\infty)$ is continuous.
We are done with the proof.
\end{proof}

The following important semi-local properties were first discovered by Knieper \cite{Kn0} (see also Knieper \cite{Kn1,Kn2,Kn3}) in the setting of rank $1$ manifolds with nonpositive curvature. Based on the results in Section \ref{geometric} and Subsection \ref{Busemann1},
we can extend them to rank $1$ manifolds without focal points by following Knieper's argument.

\begin{proposition}\label{pro13}
Let $X$ be a simply connected rank $1$ manifold without focal points which has a compact quotient, and $\{\mu_{p}\}_{p\in X}$ be an $h$-dimensional Busemann density on $X(\infty)$. Then there exist positive constants $R$ and $l$ such that for all $r>R$,
\begin{enumerate}
\item[(1)] $\mu_{p}(pr_{x}B(p,r))\geq l$, for all $p \in X$ and $x \in \overline{X}$,
where $B(p,r)$ is the open ball of radius $r$, centered at $p$;
\item[(2)] there is a positive constant $a=a(r)$ such that for any $x \in X \setminus \{p\}$ and $\xi=\gamma_{p,x}(-\infty)$,
we have $$\frac{1}{a} \cdot e^{-h \cdot d(p,x)} \leq \mu_{p}(pr_{\xi}B(x,r))\leq a\cdot e^{-h \cdot d(p,x)};$$
\item[(3)] there is a positive constant $b=b(r)$ such that for any $x \in X$,
$$\frac{1}{b} \cdot e^{-h \cdot d(p,x)} \leq \mu_{p}(pr_{p}B(x,r))\leq b\cdot e^{-h \cdot d(p,x)}.$$
\end{enumerate}
\end{proposition}

\begin{proof}
We prove item (1) first. Items (2) and (3) are consequences of (1). We already know, by Proposition \ref{pro12},
that $\text{supp}(\mu_{p}) = X(\infty)$ for all $p\in X$.
This means that for any $h$-dimensional Busemann density $\mu_p$ and any open set $U\subset X(\infty)$, $\mu_p(U)>0$.
Specifically, for any $v\in SX$ and $\epsilon>0$,
we have $\mu_p(C_{\epsilon}(v))>0$, where the open subset $C_{\epsilon}(v)\subset X(\infty)$ is defined in Section \ref{geometric}.

Now we consider the projection $pr_{\xi}:X\rightarrow X(\infty)$, for $\xi\in X(\infty)$. We have the following lemma:
\begin{lemma}\label{mu pr Bxr}
Suppose $\xi\in X(\infty)$. Then for any $x\in X$  there is an $R=R(x)>0$ such that for any $r>R$, $pr_{\xi}(B(x,r))$ contains an open subset of $X(\infty)$.
\end{lemma}
\begin{proof}
We know each element in $\Gamma$ is axial since $M=X/\Gamma$ is compact (\cite{CS} Lemma 2.1).
By Corollary \ref{co2.11}, there is a rank $1$ geodesic $c$ with $c(-\infty)=\xi$. Moreover, Proposition \ref{pro6} implies that for any $\epsilon>0$, there is an open neighborhood $V_{\epsilon}$ of $c(+\infty)$ such that for any $\eta\in V_{\epsilon}$, the unique geodesic $\gamma$ connecting $\xi$ and $\eta$ ($\gamma(-\infty)=\xi$, $\gamma(+\infty)=\eta$) satisfies that $d(c(0),\gamma)<\epsilon$. For any $x\in X$, let $R(x)=d(x,c(0))+\epsilon$, then whenever $r>R(x)$, $V_{\epsilon}\subset pr_{\xi}(B(x,r))$.
\end{proof}

Next we study the projection $pr_x :X\rightarrow X(\infty)$, for general $x\in \overline{X}=X\cup X(\infty)$.
Fix a compact fundamental domain $\mathcal{F}\subset X$ and a point $x_0\in \mathcal{F}$. Consider the set $pr_x(B(p,r))$ where $p\in \mathcal{F}$. Combining Lemma \ref{mu pr Bxr} with Proposition \ref{pro3.2}, we get the following conclusion:
\begin{lemma}
There exist positive constants $\epsilon$ and $R$, such that for any $p\in\mathcal{F}$ and any $x\in\overline{X}$, there is a vector $v\in S_{x_0}X$ such that $C_{\epsilon}(v) \subset pr_{x}B(p,R)$.
\end{lemma}
The proof of this lemma on rank $1$ manifolds with nonpositive curvatures can be found in Knieper \cite{Kn0}, p.~762. Based on Lemma \ref{mu pr Bxr}, on can check that the above result also holds for rank $1$ manifolds without focal points, by following the same argument. So we omit the proof here.

Now we claim that, for any $\epsilon>0$, there is an $l=l(\epsilon)>0$ such that for all $p\in \mathcal{F}$ and $v \in S_{x_{0}}X$,
we have $\mu_{p}(C_{\epsilon}(v))>l$. This can be proved by contradiction. We know that since $\text{supp}(\mu_{p}) = X(\infty)$, then $\mu_p(C_{\epsilon}(v))>0$, for all $p\in X$,  $v\in S_{x_0}X$ and $\epsilon>0$. Assume that there is a number $\epsilon_{0}>0$, and
sequences $\{v_n\}\subset S_{x_0}X$ and $\{p_{n}\}\subset \mathcal{F}$ with $v_n\rightarrow v$,
such that  $\mu_{p_{n}}(C_{\epsilon_{0}}(v_n))<\frac{1}{n}$. Since $v_n\rightarrow v$, there is a constant $N>0$ such that if $n>N$,
then $C_{\epsilon_{0}}(v_n)\supset C_{\frac{\epsilon_{0}}{2}}(v)$. Thus for any $p \in \mathcal{F}$, we have
$$0\leq \mu_{p}(C_{\frac{\epsilon_{0}}{2}}(v))\leq e^{h\cdot d(p,p_{n})}\mu_{p_{n}}(C_{\frac{\epsilon_{0}}{2}}(v))
\leq \frac{1}{n}e^{h\cdot \text{diam}(\mathcal{F})}, ~~n \geq N.$$
This implies that $\mu_p(C_{\frac{\epsilon_{0}}{2}}(v))=0$, which is a contradiction.

Based on the discussion in the above, we know that there exist $R>0$ and $l>0$ such that if $r>R$, then $$\mu_p(pr_x B(p,r))\geq \mu_p(C_{\epsilon}(v))\geq l,$$ for all $x\in \overline{X}$ and $p\in\mathcal{F}$.
In general, for any $p\in X$ we can always find an isometry $\beta\in\Gamma$ such that $\beta p\in \mathcal{F}$. Then by the $\Gamma$-invariance of $\mu_p$, for all $r>R$, we have $$\mu_{p}(pr_{x}B(p,r))=\mu_{\beta p}(pr_{\beta x}B(\beta p,r))>l.$$ We are done with the proof of (1).

Now we prove statement (2). It is obvious that for all $\eta\in pr_{\xi}B(x,r)$,
$$b_x(p,\eta)\leq d(x,p),~\forall~ p\in X.$$
Therefore $-h b_x(p,\eta)\geq-hd(x,p)$. So by the definition of an $h$-dimensional Busemann density we have
\begin{eqnarray*}
\mu_p(pr_{\xi}B(x,r))&=&\int_{pr_{\xi}B(x,r)}d\mu_p(\eta)=\int_{pr_{\xi}B(x,r)}e^{-hb_x(p,\eta)}d\mu_x(\eta)\\
&\geq&\int_{pr_{\xi}B(x,r)}e^{-hd(x,p)}d\mu_x(\eta)=e^{-hd(x,p)}\cdot\mu_x(pr_{\xi}B(x,r))\\
&\geq& l(r)e^{-hd(x,p)}.
\end{eqnarray*}
Here the last inequality follows from Proposition \ref{pro13}(1).

Moreover, one can check that for all $p\in X$ and $\eta\in pr_{\xi}B(x,r)$,  $$|b_x(p,\eta)-d(x,p)|\leq 2r.$$ This implies that $b_x(p,\eta)+2r\geq d(x,p)$. Then we have
\begin{eqnarray*}
\mu_p(pr_{\xi}B(x,r))&=&\int_{pr_{\xi}B(x,r)}d\mu_p(\eta)=\int_{pr_{\xi}B(x,r)}e^{-hb_x(p,\eta)}d\mu_x(\eta)\\
&\leq&\int_{pr_{\xi}B(x,r)}e^{-h(d(x,p)-2r)}d\mu_x(\eta)\leq e^{(h+1)K}e^{2rh}\cdot e^{-hd(x,p)},
\end{eqnarray*}
where $e^{(h+1)K}$ is the uniform upper bound of $\{\mu_{p}\}_{p \in X}$ in Proposition \ref{pro11}.
Let $a(r)=\max\{\frac{1}{l(r)},e^{(h+1)K}e^{2rh}\}$, we get $\frac{1}{a} \cdot e^{-h \cdot d(p,x)} \leq \mu_{p}(pr_{\xi}B(x,r))\leq a\cdot e^{-h \cdot d(p,x)}$. We are done with the proof of (2). Statement (3) can be proved in the same way, so we omit its proof here.
\end{proof}

\subsection{Uniqueness of the Busemann density}
We restate Theorem B below. It asserts the uniqueness of the Busemann density, which was originally proved by Knieper in \cite{Kn0} under the assumption of nonpositive curvature. One can check that, based on our discussion in the above, Knieper's proof also works in the situation of no focal points.

\begin{TheoremB}\label{ThmB}
Let $M=X/\Gamma$ be a compact rank $1$ manifold without focal points, then up to a multiplicative constant, the Busemann density is unique, i.e., the Patterson-Sullivan measure is the unique Busemann density.
\end{TheoremB}

We sketch the idea of the proof, following closely \cite{Kn0} $P_{764}$-$P_{767}$.
First, we fix a point $p_{0}$. Fix an arbitrary point $p \in X \setminus p_{0}$, let $\xi=pr_{p_{0}}p$.
For any positive number $\rho > R$ (where $R$ is the constant in Propsition \ref{pro13}), we call the set
$$B^{\rho}_{r}(\xi):=pr_{p_{0}}(B(p,\rho)\cap S(p_{0},d(p,p_{0})))$$
a ball of radius $r=e^{-d(p,p_{0})}$ at infinity point $\xi$. By the triangle inequality we know $pr_{p_{0}}B(p,\rho) \subseteq B^{2\rho}_{r}(\xi)$.
Then by Proposition \ref{pro13}(3) and Proposition \ref{pro11}, for $\rho > 2R$, there exists a constant number $b(\rho) > 1$ such that
$$\frac{1}{b(\rho)}e^{-h\cdot d(p,p_{0})} \leq \mu_{p_{0}}(B^{\rho}_{r}(\xi)) \leq b(\rho)e^{-h\cdot d(p,p_{0})}.$$
Next, for any subset $A \subseteq X(\infty)$
we can define the $h$-dimensional Hausdorff measure $H^{h}(A)$ of $A$
and Hausdorff dimension $Hd(A)$ of $A$ as Definition 4.5 in \cite{Kn0}.
Then we can show that there exists a constant number $c > 1$ satisfying the following estimate
$$\frac{1}{c}\mu_{p_{0}}(A) \leq H^{h}(A)\leq c\cdot \mu_{p_{0}}(A),  ~ \forall ~\mbox{Borel~set} ~A\subseteq X(\infty).$$
In the proof of the above estimations, Knieper used the property that the distance function of two geodesics on a
simply connected nonpositively curved manifold $\widetilde{M}$ (usually we call it a Hadamard manifold) is convex,
i.e., if $c_{1}$ and $c_{2}$ are two unit speed geodesics on a Hadamard manifold $\widetilde{M}$, then the function
$t\mapsto d(c_{1}(t),c_{2}(t))$ is a convex function. Although the convexity is no longer valid for manifolds without focal points, we can still get the same estimations based on
Lemma \ref{OS}. Then we can prove that the volume entropy $h$ equals the $Hd(X(\infty))$, i.e., the Hausdorff dimension of $X(\infty)$.
Finally, we can show that all Busemann densities have to coincide up to a constant.

A natural question is to study the existence and unqueness of Busemann density for rank $1$ manifolds without conjugate points.
In our proof of Busemann density for rank $1$ manifolds without focal points, Proposition \ref{pro13} (1) plays a key role, which in turn follows from the geometric properties stated in Lemma \ref{pro7} and Proposition \ref{pro6}. However, we don't know whether Lemma \ref{pro7} and Proposition \ref{pro6} are still true under the assumption of no conjugate points.

\section{\bf{Exponential volume growth rate of geodesic spheres}}\label{volume}
In this section, we consider the growth rate of the geodesic sphere $S(x,r)$ about $x$ with radius $r$, on the universal cover $X$ of a compact rank $1$ manifold $M$ without focal points. Based on the technical results prepared in Proposition \ref{pro13}, we prove Theorem C in this section. We remark that it is a generalization of Theorem 5.1 in \cite{Kn0}.


\begin{proof}[Proof of Theorem C]
Since we are interested in only sufficiently large $r$, we can assume $r>3R$, where $R$ is the constant in Proposition \ref{pro13}. Let $\{q_{1},q_{2},\cdots,q_{k}\}$ be a maximal $3R-$separated set of $S(x,r)$. Then we have that
\begin{equation}\label{eq10-1}
S(x,r)\subset \bigcup^{k}_{i=1}B(q_{i},3R),
\end{equation}
and
\begin{equation}\label{eq10-2}
B(q_{i},R) \cap B(q_{j},R) = \emptyset, ~\forall~i \neq j.
\end{equation}
By Proposition \ref{pro13} (3), there is a constant $b=b(r)$ such that for all $\rho \geq R$, we have
$$\frac{1}{b} \cdot e^{-h \cdot r} \leq \mu_{x}(pr_{x}B(q_{i},\rho))\leq b\cdot e^{-h \cdot r}, ~~i=1,2,\cdots,k.$$
Notice that $X(\infty)=pr_x(S(x,r))$. Combining with \eqref{eq10-1}, we have that
$$\mu_x(X(\infty))\leq\sum_{i=1}^k\mu_x(pr_x(S(x,r)\cap B(q_i,3R))) \leq k b\cdot e^{-h \cdot r}.$$
Also by \eqref{eq10-2}, we know that
$$\mu_x(X(\infty))\geq \sum_{i=1}^k\mu_x(pr_x(S(x,r)\cap B(q_i,R))) \geq  \frac{k}{b}\cdot e^{-h \cdot r}.$$
From the two inequalities in the above we know that there is a constant $c>0$ such that the following inequality holds
$$\frac{1}{c} \cdot e^{h \cdot r} \leq k \leq c\cdot e^{h \cdot r}.$$

Since $M$ is a compact manifold, we can see that there is a positive constant $A$, such that for all $r > 3R$ and  $R \leq \rho \leq 3R$, we have
$$\frac{1}{A} \leq \text{Vol}(B(q_{i},\rho)\cap S(x,r))\leq A, ~~i=1,2,\cdots,k.$$
This implies that
$$\frac{1}{cA}e^{h \cdot r}\leq\sum^{k}_{i=1}\text{Vol}(B(q_{i},R)\cap S(x,r)) \leq \text{Vol}S(x,r)
\leq \sum^{k}_{i=1}\text{Vol}(B(q_{i},3R)\cap S(x,r))\leq cA e^{h \cdot r}.$$
Therefore, we get
$$\frac{1}{cA} \leq \frac{\text{Vol}S(x,r)}{e^{h \cdot r}} \leq cA.$$
Let $a=cA>0$ and $r_0=3R$, then we are done with the proof.
\end{proof}

Similar to \cite{Kn0}, we can draw the conclusion that the Poincar\'e series diverges at $s=h$, based on the estimation growth rate of $\text{Vol}S(x,r)$ in Thoerem C. This is the following corollary:

\begin{corollary}\label{divergent type}
Suppose $M=X\setminus \Gamma$ is a compact rank $1$ Riemannian manifold without focal point. Then the Poincar\'e series defined in \eqref{poincare series} is divergent at $s=h$.
\end{corollary}

The proof is the same as the proof of Corollary 5.2 in \cite{Kn0}, so we omit it here.

\section{\bf{Existence and uniqueness of the measure of maximal entropy}}\label{maximal}

Let $P:SX \rightarrow X(\infty) \times X(\infty)$ be the projection given by $P(v)=(\gamma_{v}(-\infty),\gamma_{v}(+\infty))$. Denote by $\mathcal{I}^{P}=P(SX)=\{P(v)\mid v \in SX\}$ the subset of pairs in $X(\infty) \times X(\infty)$
which can be connected by a geodesic. Note that the connecting geodesic may not be unique.
Fix a point $p\in X$, we can define a $\Gamma$-invariant measure $\overline{\mu}$
on $\mathcal{I}^{P}$ by the following formula
(cf.~\cite{Kn1} Lemma 2.4):
$$d \overline{\mu}(\xi,\eta) = e^{h\cdot \beta_{p}(\xi,\eta)}d\mu_{p}(\xi) d\mu_{p}(\eta),$$
where $\beta_{p}(\xi,\eta)=-\{b_{p}(q,\xi)+b_{p}(q,\eta)\}$ is the Gromov product, and $q$ is any point on a geodesic $\gamma$ connecting $\xi$ and $\eta$. It's easy to see that the function $\beta_{p}(\xi,\eta)$ does not depend on the choice of $\gamma$ and $q$.
In geometric language, the Gromov product $\beta_{p}(\xi,\eta)$ is the length of the part of the geodesic
$\gamma_{\xi,\eta}$ between the horospheres $H_{\xi}(p)$ and $H_{\eta}(p)$.
Then, $\overline{\mu}$ induces a $\phi$-invariant measure $\mu$ on $SX$ with
\begin{equation}\label{e:def of max entropy measure}
\mu(A)=\int_{\mathcal{I}^{P}} \text{Vol}\{\pi(P^{-1}(\xi,\eta)\cap A)\}d \overline{\mu}(\xi,\eta),
\end{equation}
for all Borel sets $A\subset SX$. Here $\pi : SX \rightarrow X$ is the standard projection map and
Vol is the induced volume form on $\pi(P^{-1}(\xi,\eta))$.
By the definition of $P$, we know that if there is no geodesic connecting $\xi$ and $\eta$,
then $P^{-1}(\xi,\eta)=\emptyset$. If there are more than one geodesics connecting $\xi$ and $\eta$,
and one of them has rank $k\geq 1$, then by the flat strip theorem, all of these connecting geodesics have rank $k$.
Thus, $P^{-1}(\xi,\eta)$ is exactly the $k$-flat submanifold connecting $\xi$ and $\eta$,
which is consisting of all the rank $k$ geodesics connecting $\xi$ and $\eta$. Especially,
when $k$=1, $P^{-1}(\xi,\eta)$ is exactly the rank 1 (thus unique) geodesic connecting $\xi$ and $\eta$.
From the above we can see the induced volume form Vol satisfies that for any Borel set $A\subset SX$ and $t\in \mathbb{R}$,
$\text{Vol}\{\pi(P^{-1}(\xi,\eta)\cap \phi^{t}A)\}=\text{Vol}\{\pi(P^{-1}(\xi,\eta)\cap A)\}$.
Therefore the $\Gamma$-invariance of $\overline{\mu}$ leads to the $\Gamma$-invariance of $\mu$.

By Theorem \ref{th3.15} we know the horospheres depend continuously on their centers at $X(\infty)$.
While the value of the Gromov product is the length of the segment of the connecting geodesic between the horospheres,
 $\beta_{p}(\cdot,\cdot): X(\infty) \times X(\infty) \rightarrow \mathbb{R}$ is a continuous function.
Since  $X(\infty) \times X(\infty)$ is compact (Proposition \ref{pro2.4}), $\beta_{p}(\cdot,\cdot)$ is bounded function.
By the fact that $\mu_{p}$ is uniformly bounded and $M$ is compact,
$\mu$ can be projected to a finite $\phi$-invariant measure on $SM$.
To simplify the notation, we also use $\mu$ to denote this projected measure. If $\mu(SM)\neq 1$, we can normalize it.
So in this paper we always assume $\mu$ is a probability measure.

Now, we show that the measure $\mu$ on $SM$ is the unique maximal entropy measure.
Thus $\mu$ is independent on $p$.
This measure is called the \emph{Knieper measure} since it is first obtained by G.~Knieper on rank $1$ manifolds (cf.~\cite{Kn1}). We decompose our proof into $3$ parts: First, we show that $h_{\mu}(\phi)=h_{\text{top}}(g)$, so $\mu$ is a measure of maximal entropy; Second, we show $\mu$ is an ergodic measure; In the last step, we show that the entropy of any invariant probability measure orthogonal to $\mu$ is strictly less than $h_{\mu}(\phi)$, this leads to the uniqueness of the invariant probability measure of maximal entropy. We should remark that this argument follows the idea of Knieper in \cite{Kn1}.

\subsection{Entropy of the Knieper measure}\label{compute entropy}
In this subsection, we will calculate the entropy of the Knieper measure $\mu$ and prove that
\begin{equation}\label{hmu equal h}
h_{\mu}(\phi) = h
\end{equation}
where $h=h_{\text{top}}(g)$ is the topological entropy of the geodesic flow. We recall that the measure theoretical entropy $h_{\mu}(\phi)$  of the geodesic flow is defined to be the entropy of $\mu$ with respect to the time-$1$ map $\phi^1$, i.e. $h_{\mu}(\phi) := h_{\mu}(\phi^1)$.
Since by the variational principle, we know that $$h_{\text{top}}(g)=\sup_{\mu\in\mathcal{M}(\phi^1)}h_{\mu}(\phi^1),$$ where $\mathcal{M}(\phi^1)$ denotes the set of all $\phi^1$-invariant measures. Then, \eqref{hmu equal h} implies that $\mu$ is an invariant measure of maximal entropy.

For each $k\in \mathbb{Z}^+$, we define a metric $d_{k}$ on $SM$ or $SX$ defined by the following formula:
$$d_{k}(v,w)=\max \{d(\gamma_{v}(t),\gamma_{w}(t))\mid t \in [0,k]\}.$$
Specifically
$d_{1}(v,w)=\max \{d(\gamma_{v}(t),\gamma_{v}(t))\mid t \in [0,1]\}$, we call it \emph{Knieper metric}.
One can check that the Knieper metric $d_1$ is equivalent to the standard Sasaki metric on $SM$.

Consider a finite measurable partition $\mathcal{A}=\{A_{1},...,A_{n}\}$ of  $SM$,
satisfying that the radius of every $A_k$ under Knieper metric $d_1$ is less than $\epsilon < \min\{\text{Inj}(M),R\}$,
where $R>0$ is the constant in Proposition \ref{pro13}.
Let $\mathcal{A}^{(n)}_{\phi}:=\vee^{n-1}_{i=0}\phi^{-i}\mathcal{A}$, $n=1,2,\cdots$.
Obviously $\mathcal{A}_{\phi}^{(n)}$ is also a finite measurable partition of $SM$.
For each $v\in SM$ and $\alpha \in \mathcal{A}^{(n)}_{\phi}$, we know that if  $v\in\alpha$£¬then
$$\alpha \subset \bigcap^{n-1}_{k=0}\phi^{-k}B_{d_{1}}(\phi^{k}v,\epsilon),$$
where $B_{d_{1}}(\phi^{k}v,\epsilon)$ denotes the $\epsilon$-neighborhood of $\phi^{k}v\in SM$ under the Knieper metric $d_1$.
Similar to Lemma 2.5 in \cite{Kn1}, using Proposition \ref{pro13}(2) we can show that, on a compact rank $1$ manifold without focal points, there is $a > 0$ such that for any $\alpha \in A^{n}_{\phi}$, we have $\mu(\alpha)\leq e^{-hn}a$.
Therefore
$$H_{\mu}(\mathcal{A}^{(n)}_{\phi})=-\sum_{\alpha \in \mathcal{A}^{(n)}_{\phi}}\mu(\alpha)\log\mu(\alpha)\geq n\cdot h-\log a,$$
where $h=h_{\text{top}}(g)$ is the topological entropy of the geodesic flow.
This implies that $$h_{\mu}(\phi^{1})\geq h_{\mu}(\phi^1,\mathcal{A})=\lim_{n\rightarrow \infty}\frac{1}{n}H_{\mu}(\mathcal{A}^{(n)}_{\phi})\geq h,$$ and then $h_{\mu}(\phi^{1}) = h$. Thus we know that $\mu$ is an invariant measure of maximal entropy.

\subsection{Ergodicity of the Knieper measure}\label{proof of ergodicity}
In this subsection, we will prove that the Knieper measure $\mu$ is an ergodic measure of the geodesic flow.
Our proof follows the idea of the proof in Section 4 of Knieper \cite{Kn1}. To simplify the notation,
we also use $\mu$ to denote the $\Gamma$-invariant measure on $SX$ defined in \eqref{e:def of max entropy measure},
which is exactly the lifting of the maximal entropy measure $\mu$ we are considering.

For each $v \in SX$, let $\xi = \gamma_{v}(+\infty)$. Let
$$W^{s}(v):= \{w \in SX \mid \omega = -\nabla b_{\pi v}(q,\xi) ~~\mbox{and} ~~ b_{\pi v}(q,\xi)=0\}$$
denote the strong stable manifold through $v$.

\begin{lemma}\label{lem7}
Assume that $v \in SX$ is a recurrent rank $1$ vector, then
$$d_{1}(\phi^{t}(v),\phi^{t}(w))\rightarrow 0, ~~~t\rightarrow +\infty,$$
for any $w \in W^{s}(v)$, where $d_{1}$ is the Knieper metric defined in Subsection \ref{compute entropy}.
\end{lemma}
\begin{proof}
First of all, since $v$ is an recurrent vector on $SX$, then we can find a sequence $\{\alpha_{n}\}^{\infty}_{n=1}\subset \Gamma$ and a sequence of time $\{t_{n}\}^{\infty}_{n=1}$ with $t_{n}\rightarrow +\infty$, such that
\begin{equation}\label{rec v}
d\alpha_{n}(\phi^{t_n}v)\rightarrow v,~~~n\rightarrow +\infty.
\end{equation}

Now take an arbitrary vector $w \in W^{s}(v)$.
In \cite{Ru2},
R. Ruggiero showed that for manifold without focal points, $\gamma_{v}$ is positively asymptotic to $\gamma_{w}$.
Then by Lemma \ref{OS}, the function
$$t \longrightarrow d_{1}(\phi^{t}(w), \phi^{t}(v))$$ is monotonely non-increasing. Now we show that  $d_{1}(\phi^{t}(w), \phi^{t}(v))\rightarrow 0$ as $t\rightarrow +\infty$, by contradiction. Assume that there is a constant $c>0$ such that
$$d_{1}(\phi^{t}(w), \phi^{t}(v))\geq c, ~~~\forall~t \geq 0.$$
Then for each $t_n$ in \eqref{rec v}, we write $t=t_n+s>0$ where $s \in [-t_{n},+\infty)$, then we have
$$d_{1}(\phi^{t_{n}+s}(w),\phi^{t_{n}+s}(v))>c, ~~\forall~s \in [-t_{n},+\infty).$$
Then,
\begin{eqnarray*}
c\leq d_{1}(\phi^{t_{n}+s}(w),\phi^{t_{n}+s}(v))
& = & d_{1}(d\alpha_{n}\circ\phi^{t_{n}+s}w,d\alpha_{n}\circ\phi^{t_{n}+s}v)\\
& = & d_{1}(\phi^{s}\circ d\alpha_{n}\circ\phi^{t_{n}}w,\phi^{s}\circ d\alpha_{n}\circ\phi^{t_{n}}v).
\end{eqnarray*}

Note that $d_{1}(\phi^{t_{n}}(w),\phi^{t_{n}}(v))\leq d_1(v,w)$, and $d\alpha_{n}(\phi^{t_{n}}v)\rightarrow v$.
So we know that for any $\epsilon>0$ there is a number $N>0$ such that for all integer $n>N$, $$d_1(v,d\alpha_{n}(\phi^{t_{n}}w))\leq d_1(v,w)+\epsilon.$$ This implies that the set $\{d\alpha_{n}(\phi^{t_{n}}w)\}$ has an accumulate point. Without loss of generality, we assume $d\alpha_{n}(\phi^{t_{n}}w)\rightarrow w'$. Let $n\rightarrow +\infty$, we get that
$$c \leq  d_{1}(\phi^{s}(w'),\phi^{s}(v)) \leq d_{1}(v,w), ~~~\forall~s \in \mathbb{R}.$$
This means that the geodesics $\gamma_{v}$  and $\gamma_{w'}$ bound a flat strip,
which contradicts the assumption that $\text{rank}(v)=1$. We are done with the proof.
\end{proof}

Let $a$ be an arbitrary rank $1$ axis, and $U,V\subset X(\infty)$ be a pair of neighborhoods of  $a(+\infty)$ and $a(-\infty)$ satisfying
Proposition \ref{pro6}. Therefore for each pair  $(\xi,\eta)\in U \times V$, there is a unique rank $1$ geodesic connecting $\xi$ and $\eta$.

Let
$$\mathcal{G}(U,V)=\{c \mid c~ \mbox{is a geodesic with} ~c(-\infty)\in U ~\mbox{and} ~c(+\infty)\in V\},$$
and
$$\mathcal{G}_{rec}(U,V)=\{c \mid c\in \mathcal{G}(U,V),~\mbox{and} ~c' ~\mbox{is recurrent}\}.$$
Moreover, let $\mathcal{G}'(U,V)=\{c'(t) ~|~ c\in \mathcal{G}(U,V),~t\in\mathbb{R} \}$ and $\mathcal{G}'_{rec}(U,V)=\{c'(t) ~|~ c\in \mathcal{G}_{rec}(U,V),~t\in\mathbb{R} \}$. Then,
by Poincar\'e recurrence theorem, $\mu(\mathcal{G}'(U,V)\setminus \mathcal{G}'_{rec}(U,V))=0$.

Let $f:SX \longrightarrow \mathbb{R}$ is an arbitrary continuous $\Gamma$-invariant function.
By Birkhoff ergodic theorem,
for $\mu$-almost all vectors $v\in SX$, $f^+(v)$ and $f^-(v)$ in the following are well-defined and equal:
$$f^{+}(v):= \lim_{T\rightarrow +\infty}\frac{1}{T}\int^{T}_{0}f(\gamma'_v(t))dt,~~\mbox{and}~~ f^{-}(v):= \lim_{T\rightarrow +\infty}\frac{1}{T}\int^{T}_{0}f(\gamma'_v(-t))dt.$$
Where $\gamma_v$ is the unique geodesic with $\gamma'_v(0)=v$. It is easy to see that $f^{\pm}$ are constants along each geodesics.
So we can define $f^{\pm}(c):=f^{\pm}(c'(0))$ for every geodesic $c$.
Let $$\mathcal{\widetilde{G}}_{rec}(U,V)=\{c \mid c \in \mathcal{G}_{rec}(U,V), ~f^+(c)=f^-(c)\},$$ and $\mathcal{\widetilde{G}}'_{rec}(U,V)=\{c'(t)~|~c\in\mathcal{\widetilde{G}}_{rec}(U,V),~t\in\mathbb{R}\}$.
Therefore we have $$\mu(\mathcal{G}'(U,V)\setminus \mathcal{\widetilde{G}}'_{rec}(U,V))=0.$$

\begin{lemma}\label{Gxi}
There is a geodesic $c_1\in \mathcal{\widetilde{G}}_{rec}(U,V)$ with $c_1(-\infty)=\xi_1\subset U$, such that
$$G_{\xi_1}:= \{\eta \in V \mid \exists~c\in\mathcal{\widetilde{G}}_{rec}(U,V) ~\mbox{with}~ c(+\infty)=\eta, c(-\infty)=\xi_1\}\subset V$$
has full measure w.r.t. $\mu_{p}$, i.e., $\mu_{p}(G_{\xi_1})=\mu_{p}(V)$.
\end{lemma}
\begin{proof}
Let $\mathcal{E}=\mathcal{G}_{rec}(U,V)\setminus\mathcal{\widetilde{G}}_{rec}(U,V),$ and $\mathcal{E}'=\{c'(t)~|~c\in\mathcal{E},~t\in\mathbb{R}\}$. Obviously $\mu(\mathcal{E}')=0$. Moreover, for any $\xi\in U$, let
$$G^{co}_{\xi}= \{\eta \in V \mid  \exists~c\in\mathcal{E},~\mbox{with}~ c(-\infty)=\xi, ~c(+\infty)=\eta\}\subset V.$$
Then we have
\begin{eqnarray*}
0 = \mu(\mathcal{E}')
& = & \int_{\mathcal{I}^{P}}\text{Vol}(\pi(P^{-1}(\xi,\eta)\cap \mathcal{E}'))e^{h\beta_{p}(\xi,\eta)}d\mu_{p}(\xi) d\mu_{p}(\eta)\\
& = & \int_{P(\mathcal{E}')}\text{Vol}(\pi\circ P^{-1}(\xi,\eta)\cap \mathcal{E}')e^{h\beta_{p}(\xi,\eta)}d\mu_{p}(\xi) d\mu_{p}(\eta)\\
& \geq & \int_{\xi\in U,~\eta\in G^{co}_{\xi}}d\mu_{p}(\xi) d\mu_{p}(\eta)=\int_{\xi\in U}\left(\int_{G^{co}_{\xi}}d\mu_{p}(\eta)\right)d\mu_{p}(\xi).
\end{eqnarray*}
Therefore, $\int_{\xi\in U}(\int_{G^{co}_{\xi}}d\mu_{p}(\eta))d\mu_{p}(\xi)=0$, which implies that for  $\mu_{p}$-a.e.~$\xi \in U$,
$\mu_{p}(G^{co}_{\xi})=0$. This means that $\mu_{p}(G_{\xi})=\mu_{p}(V)$ for $\mu_{p}$-a.e.~$\xi \in U$.

Take a $\xi_1\in U$ such that there is a geodesic $c_1\in \mathcal{\widetilde{G}}_{rec}(U,V)$ with $c(-\infty)=\xi_1$, and $\mu_{p}(G_{\xi})=\mu_{p}(V)$.
From the above argument we know that the point $\xi_1$ and the geodesic $c_1$ always exist. We are done with the proof.
\end{proof}

The next lemma is a consequence of Lemma \ref{lem7} and \ref{Gxi}.
\begin{lemma}\label{lem8}
$f^{+}(c)=f^+(c_1)$ for almost all $c\in\mathcal{\widetilde{G}}_{rec}(U,V)$.
\end{lemma}
\begin{proof}
We follow the idea of the proof of Lemma 4.2 in \cite{Kn1}. Since we know that $\mu_{p}(G_{\xi_1})=\mu_{p}(V)$, then almost every $c\in \mathcal{\widetilde{G}}_{rec}(U,V)$ satisfies that $c(+\infty)\in G_{\xi_1}$. By the definition of $G_{\xi_1}$, we know that there is a geodesic $c_2\in \mathcal{\widetilde{G}}_{rec}(U,V)$ with $c_2(-\infty)=c_1(-\infty)=\xi_1$ and $c_2(+\infty)=c(+\infty)$. Then by Lemma \ref{lem7}, after a re-parametrization, we have $d_1(c_2(t),c(t))\rightarrow 0$ as $t\rightarrow +\infty$. This implies that $f^+(c)=f^+(c_2)~(=f^-(c_2))$. Similarly we have $f^-(c_2)=f^-(c_1)~(=f^+(c_1))$. Therefore we get $f^+(c)=f^+(c_1)$ for almost all $c\in\mathcal{\widetilde{G}}_{rec}(U,V)$.
\end{proof}

Now, we are ready to prove the ergodicity of $\mu$.
\begin{theorem}\label{ergodic}
The geodesic flow $\phi^{t}$ is ergodic with respect to the measure $\mu$.
\end{theorem}
\begin{proof}
It is sufficient to prove that for any continuous $\Gamma$-invariant function $f:SX \longrightarrow \mathbb{R}$, the function $f^+$ is constant almost everywhere. Let $$\widetilde{V}=\{\eta \in V \mid \exists~c \in \mathcal{G}_{rec}(U,V)~\mbox{with}~\eta=c(+\infty)\}.$$
By Lemma \ref{lem7} and Lemma \ref{lem8} we know $f^{+}(c)=f^+(c_1)$ for almost all geodesics $c$ with $c(+\infty)\in \widetilde{V}$. Let $\mathcal{E}_1 = \mathcal{G}(U,V)- \mathcal{G}_{rec}(U,V)$ and $\mathcal{E}'_1=\{c'(t) \mid c\in\mathcal{E},~t\in\mathbb{R}\}$. Moreover, for each $\xi\in U$, we define $\widetilde{V}_{\xi}^{co}:= \{\eta\in V \mid \exists~c \in \mathcal{E}_1~\mbox{with}~c(-\infty)=\xi,~c(+\infty)=\eta \}$.
Then similar to our calculation in the proof of Lemma \ref{Gxi}, we have
\begin{eqnarray*}
0 = \mu(\mathcal{E}'_1)
\geq\int_{\xi\in U}\left(\int_{\widetilde{V}^{co}_{\xi}}d\mu_{p}(\eta)\right)d\mu_{p}(\xi).
\end{eqnarray*}
Thus for $\mu_{p}$-a.e.~$\xi \in U$, $\mu_{p}(\widetilde{V}^{co}_{\xi})=0$. Let $\widetilde{V}_{\xi}=\{\eta\in V \mid  \exists~c \in \mathcal{G}_{rec}(U,V)~\mbox{with}~ c(-\infty)=\xi,~c(+\infty)=\eta\}$.
Then for $\mu_{p}$-a.e.~$\xi \in U$, we  have $\mu_{p}(\widetilde{V}_{\xi})=\mu_{p}(V)$.
Since $\widetilde{V}_{\xi} \subset \widetilde{V} \subset V$,
We know $\mu_{p}(\widetilde{V})=\mu_{p}(V)$.

Let $Y = \bigcup_{\alpha \in \Gamma}\alpha(\widetilde{V})$. Recall that $U,~V$ are neighborhoods of $a(-\infty)$ and $a(+\infty)$ respectively, where $a$ is a rank $1$ axis. Then by Lemma \ref{pro7},
one can see that $\mu_{p}(Y)=\mu_{p}(X(\infty))$.
This implies that $Z := \{c'(t) \mid c(+\infty)\in Y, ~t\in\mathbb{R}\}$ is a $\mu$-full measure set in $SX$. Since $f$ is a $\Gamma$-invariant function,
$f^{+}$ is also $\Gamma$-invariant. Therefore,
 $f^{+}=f^+(c_1)$ $\mu$-a.e.~on $Z$. This implies that $f^+$ is constant $\mu$-almost everywhere on $SX$.
 So the geodesic flow $\phi^{t}$ is ergodic with respect to the measure $\mu$.
\end{proof}

\subsection{Uniqueness of the maximal entropy measure of geodesic flow}\label{proof of uniqueness}

In this subsection, we will prove the measure $\mu$ is the unique maximal entropy measure, i.e.~for any $\phi$-invariant probability measure $\nu\neq\mu$, we have $h_{\nu}(\phi)<h_{\mu}(\phi)=h_{\text{top}}(g)$.

For  any $\phi$-invariant probability measure $\nu\neq\mu$, there is a unique decomposition
$$\nu = b \nu_{\text{sing}} + (1-b) \nu_{\text{abs}}, ~~  b \in [0,1],$$ where $\nu_{\text{abs}}$ is an $\phi$-invariant measure which is absolutely continuous with respect to $\mu$, and $\nu_{\text{sing}}$ is an $\phi$-invariant measure which is singular with respect to $\mu$.

We claim that $\nu_{\text{abs}}=\mu$.

Let $d\nu_{\text{abs}}(v)=f(v)d\mu(v)$ with $f\in L^1(SM,\mu)$ and $f\geq 0$. For any measurable set $A\subset SM$ and $t\in\mathbb{R}$ we have
$$\int_{A}f d\mu=\int_A d\nu_{\text{abs}}=\int_{\phi^{t}(A)}d\nu_{\text{abs}}=\int_{\phi^{t}(A)}f d\mu=\int_{A}f\circ \phi^{t} d\mu.$$
This leads to the fact that for any $t\in\mathbb{R}$,  $f\circ \phi^{t}(v)\equiv f(v)$,  $\mu$-a.e.~$v\in SM$. Since $\mu$ is ergodic, we know $f(v)=1$ for $\mu$-a.e.~$v\in SM$. Therefore $\nu_{\text{abs}}=\mu$.
Moreover, by the formula $$h_{\nu}(\phi)=bh_{\nu_{\text{sing}}}(\phi)+(1-b)h_{\mu}(\phi)
=bh_{\nu_{\text{sing}}}(\phi)+(1-b)h_{\text{top}}(g),$$ we know that, to prove the uniqueness of the maximal entropy measure, we just need to prove $h_{\nu}(\phi)<h_{\text{top}}(g)$ for all $\phi$-invariant measure $\nu$ which is singular with respect to $\mu$.
In the following, we will prove this inequality. We remark that our proof follows the idea of Knieper \cite{Kn1}.

Take a compact fundamental domain $\mathcal{F} \subset X$ and a fixed reference point $p\in\mathcal{F}$.
Let $R'>2R+\frac{1}{2}$ be a constant such that $\mathcal{F}\subset B(p,R')$,
where $R>0$ is the constant defined in Proposition \ref{pro13}. Consider the unit tangent bundle $SB(p,R')$ of $B(p,R')$.
Let $$D(x,R',R)=\{v \in SB(p,R')\mid \exists~r >0 ~~s.t. ~~\gamma_{v}(r)\in B(x,R)\}.$$
Proposition \ref{pro13} implies that for any $x\in\X$ with $d(p,x)>R'+R$ there is a constant $c'>0$ which is independent of $x$
such that $\mu(D(x,R',R))\geq c'e^{-hd(p,x)}$ (cf. \cite{Kn1}, Lemma 5.1). Then we have the following lemma:

\begin{lemma}\label{lem3}
Suppose $x\in X$ satisfies $d(p,x)\geq R'+R+n$, with $n \in \mathbb{Z}^+$. Then there is a constant $a=a(\delta,R',R)>0$ such that the cardinality of any $(d_{n},\delta)$-separated set of $D(x,R',R)$ is bounded by $a$.
\end{lemma}
\begin{proof}
Take a maximal $\frac{\delta}{8}$-separated set $\mathcal{S}_1=\{q_{1},q_{2},\cdots,q_{k}\}$ of $B(p,R')$, and a maximal $\frac{\delta}{8}$-separated set $\mathcal{S}_2=\{p_{1},p_{2},\cdots,p_{m}\}$ of $B(x,R)$.
By the compactness of $M=X/\Gamma$, $k$ and $m$ are bounded by a pair of constants $K=K(\delta,R')$ and $M=M(\delta, R)$ for all $p,~x\in X$.
Let $c_{ij}$ be the unique geodesic with $c_{ij}(0)=q_{i}$ and $c_{ij}(d(q_i,p_j))=p_{j}$.

We claim that $\{c'_{ij}(0)\}_{i=1,\cdots,k,~j=1,\cdots,m}$ is a $(d_n,\frac{\delta}{2})$-spanning set of $D(x,R',R)$.

Let $c: [0,T]\rightarrow X$ be an arbitrary geodesic with $c(0)\in B(p,R'),$  and $c(a)\in B(x,R)$.
Then by the definition of $\mathcal{S}_1$ and $\mathcal{S}_2$, we can find $q_{i}\in \mathcal{S}_1$ and $p_{j}\in \mathcal{S}_2$,
such that $c(0)\in B(q_{i},\frac{\delta}{8})$ and $c(T)\in B(p_{j},\frac{\delta}{8})$.
This implies that $T-\frac{\delta}{4}\leq d(q_{i},p_{j})\leq T+\frac{\delta}{4}$,
and then $d(c(T), c_{ij}(T))\leq \frac{3\delta}{8}$. Since $M=X/\Gamma$ is a compact rank $1$ manifold without focal points, by Lemma \ref{Ka}, we know that
\begin{equation}\label{equ in lem3}
d(c(t),c_{ij}(t))\leq d(c(0),c_{ij}(0))+d(c(T),c_{ij}(T))\leq \frac{\delta}{8}+\frac{3\delta}{8}=\frac{\delta}{2}, ~~t\in [0,T].
\end{equation}
Since $d(x,p)\geq n + R + R'$, then $T \geq n$. So \eqref{equ in lem3} implies that $$D(x,R',R)\subset \bigcup_{i=1,\cdots,k,~j=1,\cdots,m} B_{d_{n}}(c'_{ij}(0),\frac{\delta}{2}).$$
This means $\{c'_{ij}(0)\}_{i=1,\cdots,k,~j=1,\cdots,m}$ is a $(d_n,\frac{\delta}{2})$-spanning set of $D(x,R',R)$. Its cardinality is $km\leq K(\delta,R')\cdot M(\delta,R)$. Then the cardinality of a  $(d_n,\delta)$-separated set $D(x,R',R)$  is also bounded by $a:=K(\delta,R')\cdot M(\delta,R)$, which depends only on $\delta$, $R$, and $R'$.
\end{proof}

Let $r(n)=2n+R+R'$. We use $S(p,r(n))$ to denote the geodesic sphere centered at $p$ with radius $r(n)$.
Suppose $\{y_{1},y_{2},\cdots,y_{k(n)}\}$ is a maximal $2R$-separated set of the geodesic sphere $S(p,r(n))$.
Thus the balls $\{B(y_{i},R)\mid i=1,2,...,k(n)\}$ are pairwise disjoint and
$$S(p,r(n))\subset \bigcup^{k(n)}_{i=1}B(y_{i},2R).$$
These results imply that $\{D(y_{i},R',R)\}^{k(n)}_{i=1}$ are also pairwise disjoint,
and $$SB(p,R')\subset \bigcup^{k(n)}_{i=1}D(y_{i},R',2R).$$

Therefore, we can find a partition $\{F_{1}^{2n},F_{2}^{2n},\cdots,F_{k(n)}^{2n}\}$ of $SB(p,R')$, satisfying
$$D(y_{i},R',R) \subset F_{i}^{2n} \subset D(y_{i},R',2R), ~~i=1,2,\cdots,k(n).$$
Consider the standard projection $pr:SX\rightarrow SM$. Let $L_{i}^{2n}= pr(F_{i}^{2n})$, $i=1,\cdots,k(n)$.
Since the fundamental domain $\mathcal{F}\subset B(p,R')$,
we know $L^{2n} =\{L_{1}^{2n},L_{2}^{2n},\cdots, L_{k(n)}^{2n}\}$ is a cover of $SM$.

Next, similar as the Lemma 5.3 in \cite{Kn1}, we show that:
(1) $pr|_{SB(p,R')}$ is an at most $\alpha$-to-$1$ map, where $\alpha$ is a constant depending only on $R'$;
(2) there is a constant  $\beta > 0$, which is independent of $n$, such that $\mu(L_{i}^{2n}) \geq \beta e^{-2hn}$, for all $i=1,\cdots,k(n)$.

In fact, let $\alpha:=\sharp\{\gamma\in \Gamma \mid \gamma(\mathcal{F})\cap B( p,R')\neq \emptyset\}$,
then $pr$ is an at most $\alpha$-to-$1$ map. Thus for any measurable set $A\subset SB( p,R')$,
we have $\alpha \cdot \mu(pr(A)) \geq \mu(A)$. Choose $A=D(y_{i},R',R)$, then
\begin{eqnarray*}
\mu(L^{2n}_{i})& = &\mu(pr(F^{2n}_{i})) \geq \frac{1}{\alpha}\mu(F^{2n}_{i})\\
& \geq & \frac{1}{\alpha}\mu(D(y_{i},R',R))\geq \frac{1}{\alpha}c'e^{-h\cdot d(p,y_{i})}\\
& \geq & \frac{1}{\alpha}c'e^{-h\cdot (2n+R+R')}:=\beta\cdot e^{-2hn},
\end{eqnarray*}
where $c'$ appeared in the paragraph above the Lemma \ref{lem3}.

Let $\epsilon = \frac{1}{6}\mbox{Inj}(M)$. Suppose $V=\{v_{1},v_{2},\cdots,v_{m}\}$ is a maximal $(d_{2n},\epsilon)$-separated set of $SM$.
We can choose a partition $\mathcal{B}^{n}$ of $SM$ such that, for each $B \in \mathcal{B}^{n}$, there is an element $v_{i} \in V$ satisfying that
$$B_{d_{2n}}(v_{i},\frac{\epsilon}{2}) \subset B \subset B_{d_{2n}}(v_{i},\epsilon).$$
For each $L^{2n}_{i}$, $i=1,\cdots,k(n)$, it is obvious that
$$\sharp\{B \in\mathcal{B}^n \mid B\cap L^{2n}_i\neq\emptyset\}\leq \sharp\{v_j\in V \mid B_{d_{2n}}(v_{j},\epsilon)\cap L^{2n}_i\neq\emptyset\}.$$
Since the set $\{v_j\in V \mid B_{d_{2n}}(v_{j},\epsilon)\cap L^{2n}_i\neq\emptyset\}$ can be lifted to a $(d_{2n},\epsilon)$-separated set of $F^{2n}_i\subset D(y_{i},R',2R)\subset SX$, by Lemma \ref{lem3} we know that $\sharp\{v_j\in V \mid B_{d_{2n}}(v_{j},\epsilon)\cap L^{2n}_i\neq\emptyset\}\leq a(2\epsilon, R',2R)$.
Therefore we get
\begin{equation*}\label{B intersect L2n}
\sharp\{B \in\mathcal{B}^n \mid B\cap L^{2n}_i\neq\emptyset\}\leq a(2\epsilon, R',2R),~\forall~i=1,\cdots,k(n).
\end{equation*}

For each $n\in \mathbb{Z}^+$, Let $A^{n}= \{\phi^{n}L_{i}^{2n}\mid i=1,\cdots, k(n)\}$.
Based on the definition of $L^{2n}$, it is easy to check that for each $A \in A^{n}$ and $u,v \in A$, there is a continuous curve $l : [0,1]\rightarrow SM$ connecting $u$ and $v$£¬
such that for all $t \in [-n,n]$, we have $\text{length}(\pi \circ \phi^{t}(l))\leq R'$.

Following the discussion in Theorem \ref{ergodic}, we know that the regular set in $SM$ has positive $\mu$-measure.
Then by the ergodicity of geodesic flow with respect to $\mu$, we get $\mu(\text{sing})=0$.
The following technical lemma is Lemma 5.5 of \cite{Kn1}. See also Lemma 9.6 of \cite{Kn2}. It holds because we still have the flat strip lemma in no focal point case.
\begin{lemma}
Let $\nu$ be a Borel probability measure on $SM$. For a fixed constant $b>0$, let $A^n$ be
a sequence of measurable coverings of $SM$ such that
for each $n$ and $v, w\in A\in A^n$ there exists a continuous curve $\alpha: [0,1]\to SM$ with
$\alpha(0)=v, \alpha(1)=w$, and $Lg(\pi \phi^t(\alpha))\leq b$ for all $t\in [-n,n]$. Let $\Omega\neq SM$ be a set containing all
singular vectors. Then there exists a union $C_n$ of subsets of $A^n$ such that
$\nu(\Omega \bigtriangleup C_n):=\nu(C_n\setminus \Omega)+ \nu(\Omega \setminus C_n)\to 0$.
\end{lemma}

Furthermore, we can show that for any $\phi$-invariant measure $\nu$ which is singular with respect to $\mu$.
there is subset  $C_n \subset SM$ which is the union of some subsets of $A^{n}$
such that
$$\mu(C_{n})\rightarrow 0, ~~~~\nu(C_{n})\rightarrow 1, ~~~~n\rightarrow \infty.$$

Now we consider the entropy of $\nu$ mentioned above, i.e., $\nu$ is a $\phi$-invariant measure which is singular with respect to $\mu$.
Since the geodesic flows on the manifolds without focal points are entropy-expansive (cf.~\cite{LW}, and we will discuss the entropy-expansiveness in the next section),
and  $\mbox{diam}_{d_{2n}}(\mathcal{B}^n)\leq 2\epsilon=\frac{1}{3}\mbox{Inj}(M)$, by Theorem 3.5 of \cite{Bo3} and Corollary 3.4 of \cite{Kn1}, we know that $$2n\cdot h_{\nu}(\phi^{1})=h_{\nu}(\phi^{2n},\mathcal{B}^{n})\leq H_{\nu}(\mathcal{B}^{n})=-\sum_{B \in \mathcal{B}^{n}}\nu(B)\log \nu(B).$$

We divide $\mathcal{B}^n$ into two subsets: $$\mathcal{B}^n_+:=\{B\in\mathcal{B}^n \mid \phi^n B\cap C_n\neq\emptyset\},~~\mbox{and}~~ \mathcal{B}^n_{-}:=\mathcal{B}^n\setminus\mathcal{B}^n_{+}.$$  Let $a_{n}=\sum_{B\in \mathcal{B}^{n}_{+}}\nu(B)$. Obviously $1>a_n\geq\nu(C_n)\rightarrow 1$ as $n\rightarrow \infty$.  Then we have:
\begin{eqnarray*}
2n\cdot h_{\nu}(\phi^{1}) &\leq& -\sum_{B\in\mathcal{B}^n_{+}}\nu(B)\log\nu(B)-\sum_{B\in\mathcal{B}^n_{-}}\nu(B)\log\nu(B)\\
&\leq& \Big(\sum_{B\in\mathcal{B}^n_{+}}\nu(B)\Big)\log\sharp(\mathcal{B}^n_+)+
\Big(\sum_{B\in\mathcal{B}^n_{-}}\nu(B)\Big)\log\sharp(\mathcal{B}^n_{-})+\frac{2}{e}\\
&\leq& a_n\log\sharp(\mathcal{B}^n_+)+(1-a_n)\log\sharp(\mathcal{B}^n)+\frac{2}{e},
\end{eqnarray*}
where the second inequality comes from Lemma 5.7 in \cite{Kn1}. By the same calculation as the one in \cite{Kn1}, we have the following estimations:
\begin{eqnarray*}
\sharp(\mathcal{B}^n_{+})
& = & \sum_{\{L^{2n}_{i}\in L^{2n}\ \mid\ \phi^{n}L^{2n}_{i} \subset C_{n}\}}\sharp\{B\in \mathcal{B}^{n}\mid  \phi^{n}B \cap \phi^{n}L^{2n}_{i} \neq \emptyset\}\\
&\leq & a(\epsilon,R',2R)\cdot\sharp\{L^{2n}_i\mid \phi^n L^{2n}_i\subset C_n\},
\end{eqnarray*}
and
\begin{eqnarray*}
\sharp\{L^{2n}_i\mid \phi^n L^{2n}_i\subset C_n\}
& = & \sharp\{L^{2n}_i\mid  L^{2n}_i\subset \phi^{-n} C_n\}\\
&\leq & \alpha\frac{\mu(C_{n})}{\min \mu(L^{2n}_{i})}\leq \frac{\alpha}{\beta}\mu(C_{n})e^{2hn}.
\end{eqnarray*}
Thus
$$\sharp(\mathcal{B}^n_{+})\leq K\cdot\mu(C_n)e^{2hn}.$$
where $K=\frac{\alpha}{\beta}a(\epsilon,R',2R)$ is independent of $n$.

On the other hand,
\begin{eqnarray*}
\sharp\{\mathcal{B}^n\}
& = & \sharp\{B\in \mathcal{B}^n \mid \phi^{n}B\cap SM \neq \emptyset\}\\
& = & \sum_{i=1}^{\sharp L^{2n}} \sharp\{B\in \mathcal{B}^n \mid \phi^{n}B\cap \phi^{n}L^{2n}_{i} \neq \emptyset\}\\
& = & a(\epsilon,R',2R)\cdot \sharp L^{2n} \leq a(\epsilon,R',2R)\cdot \alpha \cdot \frac{\mu(SM)}{\min \mu (L^{2n}_{i})}\\
& \leq & a(\epsilon,R',2R)\cdot \frac{\alpha}{\beta}e^{2hn}=K\cdot e^{2hn}.
\end{eqnarray*}

Then we have
\begin{eqnarray*}
2n\cdot h_{\nu}(\phi^{1})
& \leq & a_n\log\sharp(\mathcal{B}^n_+)+(1-a_n)\log\sharp(\mathcal{B}^n)+\frac{2}{e}\\
& \leq &  a_n\log(K\cdot\mu(C_n)e^{2hn})+(1-a_n)\log(K\cdot e^{2hn})+\frac{2}{e},
\end{eqnarray*}
thus
$$2n(h_{\nu}(\phi^{1})-h)-\frac{2}{e}\leq a_{n}\log (K\cdot \mu(C_{n}))+(1-a_{n})\log K.$$

Since $1\geq a_{n}\geq \nu(C_{n})\rightarrow 1$ and $\mu(C_{n})\rightarrow 0$, we know the righthand side of the above inequality approaches to $-\infty$ as $n\rightarrow \infty$. Therefore we must have $$h_{\nu}(\phi^{1})< h= h_{\text{top}}(g).$$
We are done with the proof of the uniqueness of the maximal entropy measure for the geodesic flow on the manifold without focal points.

\section{\bf{Entropy gap and growth rate of singular and regular closed geodesics}}\label{closedgeodesic}

In this section, we are going to discuss the distribution of the primitive closed geodesics on compact rank $1$ manifolds without focal points. First, we will prove the entropy gap $h_{\text{top}}(\phi|_{\text{sing}})< h_{\text{top}}(g)$, which is the second part of Theorem A. Based on this result, We can show that the the ratio $\frac{P_{\text{sing}}(t)}{P_{\text{reg}}(t)}$ decays exponentially as $t\rightarrow +\infty$. This means that the set of the regular primitive  closed geodesics grows exponentially faster than the set of singular primitive closed geodesic. To show this, we need to consider the upper semi-continuity of $h_{\mu}(\phi):\mathcal{M}_{inv}(\phi)\rightarrow \mathbb{R}$ for the geodesic flow.

We use $\phi^{1}$ to denote by the time-1 map of the geodesic flow. For each $v \in SM$ and $\epsilon>0$, define
\begin{eqnarray*}
\Gamma_{\epsilon}(v)& := &\{w \in SM\mid d_{1}(\phi_{n}(v),\phi_{n}(w))\leq \epsilon, \forall~n \in \mathbb{Z}\}\\
& = & \{w \in SM\mid d(\gamma_{v}(t),\gamma_{w}(t))\leq
\epsilon, \forall~t\in \mathbb{R}\},
\end{eqnarray*}
where $d_{1}$ is the Knieper metric defined in Subsection \ref{compute entropy}.

\begin{Defi}[Entropy-expansiveness]
The geodesic flow $\phi^t: SM \rightarrow SM$ is called entropy-expansive if
$h^{\ast}_{\phi_{1}}(\epsilon):=\sup_{v \in SM}h(\phi_1,\Gamma_{\epsilon}(v))=0$.
\end{Defi}

In \cite{Bo3}, Bowen showed that if $T:X\rightarrow X$ is entropy-expansive, then the map $\mu\mapsto h_{\mu}(T)$ is upper semi-continuous. It was showed in \cite{LW} that if the compact manifold $(M,g)$ is bounded asymptote and have no conjugate points, then the geodesic flow on $SM$ in entropy-expansive. Since a manifold without focal points is always bounded asymptotic (cf.~\cite{RM}), the geodesic flow on it is entropy-expansive. Combining the two results in the above, we have the conclusion that the map $\mu\mapsto h_{\mu}(\phi)$ is upper semi-continuous. This is the following proposition.

\begin{proposition}\label{u-s-conti}
The map $\mu\mapsto h_{\mu}(\phi)$ is upper semi-continuous, for the geodesic flow $\phi^t:SM\rightarrow SM$ on a smooth connected compact Riemmannian manifold $(M,g)$ without focal points.
\end{proposition}

Based on the discussion in the above, we can prove the second part of Theorem A.

\begin{theorem}[Entropy gap]\label{th4}
Let $(M,g)$ be a compact rank $1$ Riemannian manifold without focal points. Then $$h_{\text{top}}(\phi|_{\text{sing}})< h_{\text{top}}(g),$$
where $h_{\text{top}}(g)$ denotes the topological entropy of geodesic flow on $SM$,
and $h_{\text{top}}(\phi|_{\text{sing}})$ is the topological entropy of geodesic flow restricted to the singular set \text{sing}.
\end{theorem}
\begin{proof}
We prove this theorem by contradiction. Assume that $h_{\text{top}}(\phi|_{\text{sing}})= h_{\text{top}}(g)$. Then, since the singular set sing is a closed $\phi$-invariant subset of $SM$, we can find a sequence of invariant probability  measures $\omega_{k}$ supported on sing with
$$h_{\omega_{k}}(\phi)\geq h_{\text{top}}(\phi|_{\text{sing}})-\frac{1}{k}.$$
By the compactness of $\mathcal{M}_{inv}$, the set $\{\omega_{k}\}_{k\in\mathbb{Z}^+}$ has an accumulate point $\omega$ (under the weak$^{\ast}$ topology). Obviously $$\text{supp}(\omega)\subset\text{sing}.$$ Thus $\omega \neq \mu$, where $\mu$ is the Knieper measure, since the Knieper measure is supported on the regular set.
By the upper semi-continuity of the map $\mu\mapsto h_{\mu}(\phi)$,
$$h_{\omega}(\phi)=h_{\text{top}}(\phi|_{\text{sing}})=h_{\text{top}}(g).$$
This contradicts to the uniqueness of the invariant measure of maximal entropy.
We have proved that $h_{\text{top}}(\phi|_{\text{sing}})< h_{\text{top}}(g)$.
\end{proof}

Finally, we are ready to proof Theorem D which is restated below.

\begin{TheoremD}
Let $(M,g)$ be a compact rank $1$ Riemannian manifold without focal points. Then there exist $a>0$ and $t_1>0$ such that for all $t>t_1$,
$$\frac{e^{ht}}{at}\leq P_{\text{reg}}(t)\leq a e^{ht}.$$
Moreover, there exist positive constants $\epsilon$ and $t_{2}$ such that
$$\frac{P_{\text{sing}}(t)}{P_{\text{reg}}(t)}\leq e^{-\epsilon t}, ~~t> t_{2}.$$
\end{TheoremD}
\begin{proof}
It was shown in \cite{Kn0} that, for the geodesic flow defined on a compact rank $1$ manifold with nonpositive curvature, there is a $t_1>0$ such that for all $t>t_1$,
\begin{equation}\label{Preg}
\frac{e^{ht}}{at}\leq P_{\text{reg}}(t)\leq a e^{ht},
\end{equation}
for some $a>0$. One can check that this proof is also valid for manifolds without focal points. In fact, the upper bound for $P_{\text{reg}}(t)$ is a direct consequence of Theorem C. The lower bound was proved for rank $1$ manifolds with nonpositive curvatures by Knieper in \cite{Kn0}, Theorem 5.8. This theorem relies on Lemma 5.6 and 5.7 in \cite{Kn0}, which was originally proved by Knieper in \cite{KnAM}. We will exhibits the no focal points version of these two lemmas. And then the lower bound follows. Here we define the function $l(\alpha):=\min\{d(p,\alpha(p)) \mid p\in X\}$, for any $\alpha\in\Gamma$.

\begin{lemma}\label{lem56}
Let $X$ be a simply connected manifold without focal points, $M=X/\Gamma$, and $c$ be a rank $1$ axis of some $\alpha \in \Gamma$ and $x_0=c(0)$.
Suppose $l(\alpha)=t>0$. Then for any $\varepsilon>0$
there exist a neighborhood $U\subset X(\infty)$ of $c(-\infty)$, a neighborhood $V\subset X(\infty)$ of $c(+\infty)$, and constants $\rho>0$,
$n\in\mathbb{Z}^+$, such that for all $\gamma\in \Gamma$ with $d(x_0,\gamma(x_0))\leq t$ and $\gamma(U)\cap V=\emptyset$,
the axis $a$ of $\alpha^n\gamma^{-1}\alpha^n$ is rank $1$ with end points in $U$ and $V$.
Moreover $l(\alpha^n\gamma^{-1}\alpha^n)\leq t+\rho$ and $a'(0)\in B^{d_{1}}_{\varepsilon}(c'(0))$. i.e., $\max_{t\in [0,1]}d(a(t),c(t))\leq \epsilon$.
\end{lemma}
This lemma corresponds to Lemma 5.6 in \cite{Kn0}.
\begin{proof}
This lemma is a consequence of our results in Section \ref{geometric}.
By Proposition \ref{pro6}, for a fixed rank $1$ axis $c$, consider any point $c(s)$ in $c$, we know that for any $\epsilon >0$,
there exist neighborhoods $U_{\epsilon}(c(s))$ of $c(-\infty)$ and $V_{\epsilon}(c(s))$ of $c(+\infty)$
such that for each pair $(x',y')\in U_{\epsilon}(c(s))\times V_{\varepsilon}(c(s))$,
the connecting geodesic $\gamma_{x',y'}$ exists and is rank $1$ (thus unique!), and $d(c(s),\gamma_{x',y'})<\frac{\epsilon}{4}$.
Now we take $U:=U_{\epsilon}(c(0))\cap U_{\epsilon}(c(1))$ and $V:=V_{\varepsilon}(c(0))\cap V_{\varepsilon}(c(1))$,
by Proposition \ref{pro6}, we know for any $x'\in U$ and $y'\in V$, the unique rank $1$ connecting geodesic $\gamma_{x',y'}:=a$ satisfies
$$d_{1}(a'(0),c'(0))\leq\varepsilon.$$
Thus $a'(0)\in B_{\varepsilon}^{d_{1}}(c'(0))$,
where $B_{\varepsilon}^{d_{1}}(c'(0))$ denotes the $\varepsilon$-neighborhood of $c'(0)$ in $SX$ in the sense of Knieper metric $d_{1}$.

Lemma \ref{pro7} implies that there exists a constant $n\in\mathbb{Z}^+$ such that
$$\alpha^n(X(\infty)-U)\subset V,~~\alpha^{-n}(X(\infty)-V)\subset U.$$
Note that $\overline{V}\subset X(\infty)-U$, $\overline{U}\subset X(\infty)-V$, and $\gamma(U)\cap V=\emptyset$. We know that
$$\alpha^n\gamma^{-1}\alpha^n(\overline{V})\subset\alpha^n\gamma^{-1}\alpha^n(X(\infty)-U)\subset\alpha^n\gamma^{-1}(V)\subset\alpha^n(X(\infty)-U)\subset V.$$
Similarly, we have
$$\alpha^{-n}\gamma\alpha^{-n}(\overline{U})\subset U.$$
This implies that $\alpha^n\gamma^{-1}\alpha^n$ has end points in $U$ and $V$.
Moreover,
\begin{eqnarray*}
d(x_0,\alpha^n\gamma^{-1}\alpha^n(x_0)&\leq & d(\alpha^{-n}(x_0),x_0)+d(x_0,\gamma^{-1}\alpha^n(x_0))\\
&=& d(\alpha^{-n}(x_0),x_0)+d(\gamma(x_0),\alpha^n(x_0))\\
&=& d(\alpha^{-n}(x_0),x_0)+d(\gamma(x_0),x_0)+d(x_0,\alpha^n(x_0))\\
&\leq & 2nl(\alpha)+t.
\end{eqnarray*}
Let $\rho=2nl(\alpha)+t$, then we have
$$d(x_0,\alpha^n\gamma^{-1}\alpha^n(x_0))\leq\rho+t,$$ which implies $l(\alpha^n\gamma^{-1}\alpha^n)\leq\rho+t$. We are done with the proof of the lemma.
\end{proof}

The next lemma is the no focal points version of Lemma 5.7 in \cite{Kn0}. Here we define $\Gamma_t^s(p):=\{\gamma\in \Gamma \mid s\leq d(p,\gamma(p))\leq t\}$ and $\Gamma_t(p):=\Gamma_t^0(p)$ for any $p\in X$.

\begin{lemma}\label{lem57}
Let $X$ be a simply connected manifold without focal points, and $c$ be a rank $1$ axis of some $\eta\in \Gamma$ and $x_0=c(0)$. Suppose $t=l(\eta)$. Then there exist constants $r,s>0$ such that $$\sharp\{\gamma\in\Gamma_{t+r}(x_0) \mid \gamma(U)\cap V=\emptyset\}\geq\frac{1}{4}\sharp\Gamma_t^s(x_0).$$
\end{lemma}

Similar as Lemma \ref{lem56}, this lemma is also a consequence of our results in Section \ref{geometric} (Proposition \ref{pro6} and \ref{pro8}). Because of the limitation of the length of this article, we omit the proof at here. Readers can give the proof by follow Knieper method of in \cite{KnAM} without any difficulty.

Choose a rank $1$ axis $c$ of some $\eta\in \Gamma$. Let $x_0=c(0)$ and $t_0=l(\eta)$. Then due to Lemma \ref{lem56} and \ref{lem57}, for any $\varepsilon>0$, we have $$\sharp A_{\varepsilon}(t_0+\rho+r)\geq\frac{1}{4}\sharp \Gamma_t^s(x_0),$$ where $A_{\varepsilon}(t):=\{\gamma\in\Gamma \mid l(\gamma)\leq t,~a'(0)\in B_{\varepsilon}^{d_{1}}(c'(0))~\mbox{where~a~is~rank~1~axis~of}~\gamma\}$. Together with Corollary 5.5 in \cite{Kn0}, we can get the lower bound in (\ref{Preg}). We are done with the proof of the first part of Theorem D.

To prove the second part, we will follow the idea in the proof of Corollary 6.2 in \cite{Kn1}. Suppose $h_{\text{top}}(g)=h$. Let $\epsilon_0=\frac{h-h_{\text{top}}(\phi|_{\text{sing}})}{2}$. First we consider the numerator $P_{\text{sing}}(t)= \sharp \mathcal{P}_{\text{sing}}(t)$. For each $t>0$, let $n(t)=[t]+1$ where $[t]$ is the largest integer less or qual to $t$. Let $\delta=\text{Inj}(M)$ which is the injective radius of $M$. Then it is easy to see that all the primitive closed geodesics in $\mathcal{P}_{\text{sing}}(t)$ form a $(t(n),\delta)$-separated set in sing. This leads to the inequality:
\begin{equation}\label{Psing}
\limsup_{t\rightarrow+\infty} \frac{1}{t}\log P_{\text{sing}}(t) \leq h_{\text{top}}(\phi|_{\text{sing}})< h_{\text{top}}(g)-\epsilon_0=h-\epsilon_0.
\end{equation}
Now consider denominator $P_{\text{reg}}(t)= \sharp \mathcal{P}_{\text{reg}}(t)$. We already explained that the formula (\ref{Preg}) holds for geodesic flows on compact manifolds without focal points. Combining the two inequalities \eqref{Psing} and \eqref{Preg} in the above, take $\epsilon<\epsilon_0$,
then there is a $t_2 >0$ such that for all $t>t_2$, $$\frac{P_{\text{sing}}(t)}{P_{\text{reg}}(t)}\leq e^{-\epsilon t}.$$

We are done with the proof.
\end{proof}

\section*{\textbf{Acknowledgements}}
F.~Liu is partially supported by NSFC under Grant Nos.11301305 and 11571207. F.~Wang is partially supported by NSFC under Grant No.11571387 and by the State Scholarship Fund from China Scholarship Council (CSC).
W.~Wu is partially supported by NSFC under Grant Nos.11701559 and 11571387.
All of us would like to express our great gratitude to the anonymous referee for her/his tremendous suggestions and comments on the first
version of the article.

\end{document}